\documentclass[12pt]{amsart}
\usepackage{amssymb,verbatim,amscd,amsmath,graphicx,enumerate}
\usepackage{graphicx}
\usepackage{caption}
\usepackage{subcaption}
\usepackage{pdfsync}
\usepackage{fullpage}
\usepackage{color}
\usepackage{marginnote}
\usepackage{tikz}

\usetikzlibrary{patterns, arrows, positioning}
\usetikzlibrary{calc}
\usetikzlibrary{matrix}
\usepgflibrary{arrows}
\usetikzlibrary{arrows}
\usetikzlibrary{decorations.pathreplacing}
\usetikzlibrary{decorations.pathmorphing}
\usepackage{hyperref}

\numberwithin{equation}{section}
\newtheorem{theorem}{Theorem}[section]
\newtheorem{lemma}[theorem]{Lemma}
\newtheorem{proposition}[theorem]{Proposition}
\newtheorem{corollary}[theorem]{Corollary}

\theoremstyle{definition}
\newtheorem{definition}[theorem]{Definition}
\newtheorem{algorithm}[theorem]{Algorithm}

\newtheorem{def-prop}[theorem]{Definition-Proposition}
\newtheorem{remark}[theorem]{Remark}
\newtheorem{example}[theorem]{Example}

\newtheorem*{acknowledgement}{Acknowledgements}

\newtheorem{setup}[theorem]{Set-up}
\newtheorem*{Mysketch}{Sketch of proof} 
  {\pushQED{\qed}\begin{Mysketch}}
  {\popQED\end{Mysketch}}

\DeclareMathOperator{\depth}{depth}

\DeclareMathOperator{\Ext}{Ext}

\DeclareMathOperator{\ini}{in}
\DeclareMathOperator{\pol}{pol}

\def\m{{\mathfrak m}}

\newcommand{\lcm}[1]{\ensuremath{{\rm{lcm}}{#1}}}
\newcommand{\ul}[1]{\ensuremath{\underline{#1}}}

\begin{document}

\title{Initially regular sequences and depths of ideals}

\author{Louiza Fouli}
\address{Department of Mathematical Sciences \\
New Mexico State University\\
P.O. Box 30001 \\
Department 3MB \\
Las Cruces, NM 88003}
\email{lfouli@nmsu.edu}
\urladdr{http://www.web.nmsu.edu/~lfouli}

\author{Huy T\`ai H\`a}
\address{Department of Mathematics \\
Tulane University \\
6823 St. Charles Avenue \\
New Orleans, LA 70118}
\email{tha@tulane.edu}
\urladdr{http://www.math.tulane.edu/~tai/}

\author{Susan Morey}
\address{Department of Mathematics \\
Texas State University\\
601 University Drive\\
San Marcos, TX 78666}
\email{morey@txstate.edu}
\urladdr{http://www.txstate.edu/~sm26/}

\keywords{regular sequence, depth, projective dimension, monomial ideal, edge ideal, Gr\"{o}bner basis, initial ideal}
\subjclass[2010]{13C15, 13D05, 05E40, 13F20, 13P10}

\begin{abstract}
For an arbitrary ideal $I$ in a polynomial ring $R$ we define the notion of initially regular sequences on $R/I$. These sequences share properties with regular sequences. In particular, the length of an initially regular sequence provides a lower bound for the depth of $R/I$. 
Using combinatorial information from the initial ideal of $I$ we construct sequences of linear polynomials that form initially regular sequences on $R/I$. 
We identify situations where initially regular sequences are also regular sequences, and we show that our results can be combined with polarization to improve known depth bounds for general monomial ideals.
\end{abstract}

\maketitle

\section{Introduction}\label{intro}

A fundamental invariant in commutative algebra and algebraic geometry is the depth of a module. It appears naturally in the characterization of Cohen-Macaulay rings and modules or, more generally, in Serre's criteria $(S_k)$'s (cf. \cite{BH, Serre}). The notion of depth was initially introduced as a homological invariant (under the name of \emph{homological codimension} --- see \cite{AB}). Specifically, for a finitely generated module $M$ over a local (or graded) ring $R$ with a maximal (homogenous) ideal $\m$, the depth of $M$ is
$$\depth M := \min \{d \mid \Ext^d_R(R/\m, M) \not= 0\}.$$
From duality theory, depth is also known to be closely related to local cohomology (cf. \cite{Gro}). Particularly,
$\depth M = \min\{ d \mid H^d_\m(M) \not= 0\}.$

Our work is driven by the important fact that $\depth M$ is measured by the maximum length of an $M$-\emph{regular sequence} in $\m$ (a sequence of elements $f_1, \ldots, f_d \in \m$ is said to be an $M$-regular sequence if for each $i$, $f_i$ is a non-zerodivisor on $M/(f_1, \ldots, f_{i-1})M$). Making use of a regular element (or sequence) is an essential tool in the proofs of many important results, especially when the technique involves taking hyperplane sections. In practice, however, finding a concrete description of regular sequences is a difficult task.

Our focus in this paper is on modules of the form $R/I$, where $R$ is a polynomial ring and $I \subseteq R$ is an arbitrary ideal. We introduce a new notion, called an \emph{initially regular sequence} on $R/I$, whose concrete description is tractable and whose length gives an effective lower bound for the depth of $R/I$. 

During the past two decades, many papers have appeared with various approaches to computing lower bounds for the depth, or equivalently upper bounds for the projective dimension, of $R/I$ for a \emph{squarefree} monomial ideal $I$ (cf. \cite{DS,  DS2, FH, HHKO1, LM2,  P2, P3}). The general idea has been to associate to the ideal $I$ a graph or hypergraph $H$ and use \emph{dominating} or \emph{packing} invariants of $H$ to bound the depth of $R/I$. In these works, the bounds are obtained using nonconstructive techniques that do not generally provide regular sequences.
As a consequence of our work, we provide new bounds on depth of $R/I$ for any arbitrary ideal $I$ in $R$. If, in addition, $I$ is a squarefree monomial ideal, then our bounds compare favorably to previously known combinatorial bounds. Perhaps the most interesting application of our results occurs when the generators of the initial ideal of $I$ have high degrees, in which case our bounds are usually substantial improvements over previously known bounds. Also,  our results include explicit sequences that behave similarly to regular sequences and realize the depth bounds. Furthermore, we show that polarization can be combined with our techniques to produce longer initially regular sequences that effectively compute the depth.

The motivation for our definition  comes from the fact that $\depth R/I \geq \depth R/\ini(I)$, where $\ini(I)$ is the initial ideal of $I$ with respect to any order, see \cite[Theorem 3.3.4]{HH}.  We give the following simplified definition of an initially regular sequence (see Definition~\ref{ini reg} for a more general version).

\begin{definition} \label{ini reg simple}
Let $R$ be a polynomial ring over a field and let $I \subseteq R$ be a proper ideal. A sequence of nonconstant polynomials $f_1, \ldots, f_q$ is said to be an \emph{initially regular sequence} on $R/I$ if for each $i = 1, \ldots, q$, $f_i$ is a regular element on $R/I_i$, where $I_i = \ini (I_{i-1}, f_{i-1})$ (here, by convention, $I_1 = \ini(I)$), where the initial ideals are taken with respect to a fixed monomial term order.
\end{definition}

The idea of passing to initial ideals is not new; it was even used by Macaulay in his 1927 result and was his main reason for introducing monomial orders. He proved that the Hilbert function of $R/I$ is the same as the Hilbert function of $R/\ini(I)$, see \cite[Theorem~15.26]{Eis}. More recently, Conca and Varbaro proved that $\depth R/I =\depth R/\ini(I)$, provided that $\ini(I)$ is squarefree \cite[Corollary~2.7]{CV}.

The following example illustrates our definition.

\begin{example}
Let $R=\mathbb{Q}[x_1,x_2,x_3,x_4,x_5]$. Let $I=(x_1x_2x_3+x_3x_4,x_2x_5+x_1x_2x_4,x_3x_5)$ be a polynomial ideal in $R$. Using Macaulay~2~\cite{M2} we determine that $\depth R/I=2$. Notice that $\ini(I)=(x_3x_5, x_3x_4^2, x_1x_2x_4 , x_1x_2x_3)$ with respect to the graded reverse lexicographic order with $x_1>x_2>x_5>x_4>x_3$. Let  $f=x_1+x_2$ and $g=x_5+x_3$. Then, $f,g$ (and $g, f$) is an initially regular sequence on $R/I$. In fact, $f,g$ and $g,f$ are also both regular sequences on $R/I$.
\end{example}

One can see that using Definition~\ref{ini reg simple} and the fact that $\depth R/I \geq \depth R/\ini(I)$, if $f_1, \ldots, f_q$ is an initially regular sequence on $R/I$, then $\depth R/I \ge q$,  Proposition~\ref{ini reg bound}. The task at hand is then to explicitly construct initially regular sequences. Although taking repeated initial ideals appears to be rather cumbersome, we show in one of our main results that some basic linear polynomials  will form an initially regular sequence with respect to an appropriate term order, giving both a combinatorial way to find a lower bound on the depth and a sequence of elements that in many cases is a regular sequence, and in others shares properties with one. For a monomial ideal $I \subseteq R$ and a variable $x$ of $R$, $d_x(I)$ denotes the maximum power of $x$ appearing in the minimal monomial generators of $I$. The following Theorem is a simplified version of Theorem~\ref{main}.

\begin{theorem} {\rm{(see Theorem~\ref{main})}} \label{main simple}
 Let $I$ be an ideal in a polynomial ring $R$ and $>$ a term order. Suppose that $\{b_{i,j} \mid 1 \leq i \leq q, 0 \leq j \leq t_i\}$ are distinct variables of $R$ such that $b_{i,0}> b_{i,j}$ for all $1\leq i \leq q$ and all $j$. We further assume that: 
\begin{enumerate}
\item $d_{b_{i,j}}(\ini_{>}(I)) \leq 1$, for all $i \geq 1$ and $j \geq 1$; and
\item for each $i = 1, \ldots, q$, if $M$ is a monomial generator of $\ini_{>}(I)$ and $b_{i,0}$ divides $M$, then there exists a $j \geq 1$ such that $b_{i,j}$ divides $M$.
\end{enumerate}
Let $f_i = \sum_{j=0}^{t_i} b_{i,j}$, for $1 \leq i \leq q$. Then $f_1, \ldots, f_q$ is an initially regular sequence on $R/I$.  In particular, $\depth R/I \geq q$.

\end{theorem}

In the more general version, Theorem~\ref{main}, different term orders can be used for each set of $b_{i,j}$ with $ 1\leq i \leq q$ as well as a different term order to compute $\ini_{>}(I)$. The next example shows how to utilize Theorem~\ref{main} to obtain a maximal regular sequence that realizes the depth.
\begin{example}\label{extended pentagon}
Let $I = (x_1x_2+x_1x_3, x_2x_3+x_3^2, x_3x_4, x_4x_5, x_5x_1, x_1x_6, x_5x_7, x_7x_8) \subset R=\mathbb{Q}[x_1, \ldots, x_8]$. Using Macaulay 2 \cite{M2}, we see that $\depth R/I = 3$. This value of $\depth R/I$ can also be obtained as follows.
\begin{enumerate}
\item Choose the graded reverse lexicographic order in $R$ with $x_1 > \ldots > x_8$. Then
$J = \ini(I) =(x_1x_2,x_2x_3,x_3x_4,x_4x_5,x_1x_5,x_1x_6,x_5x_7,x_7x_8)$ is the edge ideal of the the graph $H = H(J)$ depicted below.

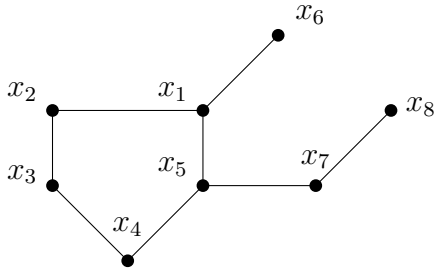
\begin{tikzpicture}

    \tikzstyle{point}=[circle,thick,draw=black,fill=black,inner sep=0pt,minimum width=4pt,minimum height=4pt]

 \node (a)[point,label={[xshift=-0.4cm, yshift=-.10 cm]: $x_1$}] at (5,2.5) {};

    \node (b)[point,,label={[xshift=-0.4cm, yshift=-.10 cm]:$x_2$}] at (3,2.5) {};

\node (d)[point,,label={[xshift=0cm, yshift=0.1 cm]: $x_4$}] at (4,0.5) {};
  \node(c)[point,,label={[xshift=-0.4cm, yshift=-.20 cm]:$x_3$}] at (3,1.5) {};

   \node (e)[point,label={[xshift=-0.4cm, yshift=-.10 cm]:$x_5$}] at (5,1.5) {};

    \node (f)[point,label={[label distance=0cm]3:$x_6$}] at (6,3.5) {};

\node (g)[point,,label={[xshift=0 cm, yshift=0 cm]: $x_7$}] at (6.5,1.5) {};
  \node(h)[point,,label={[xshift=0.4cm, yshift=-0.3 cm]:$x_8$}] at (7.5,2.5) {};

\draw (a.center) -- (b.center) -- (c.center) -- (d.center) -- (e.center) -- (a.center);
\draw (a.center) -- (f.center);
\draw (e.center) -- (g.center) -- (h.center);
\end{tikzpicture}

\item Applying Theorem~\ref{main} with any term order on $R$ in which $x_6 > x_1 > x_8 > x_7 > x_4 > x_5 > x_3 $, we get that
    $$x_6+x_1, x_8+x_7, x_4+x_5+x_3$$
    forms an initially regular sequence on $R/I$. Thus,  $\depth R/I \geq 3$ by Theorem~\ref{main}. In fact,  $x_6+x_1, x_8+x_7, x_4+x_5+x_3$ is a regular sequence on $R/J$ by Corollary~\ref{regular sums of two}. Moreover, one can check using Macaulay 2 \cite{M2} that this sequence is also regular on $R/I$ and thus realizes the depth of $R/I$. 
\end{enumerate}
\end{example}

Although Theorem~\ref{main} (see also Algorithm~\ref{alg}) is particularly easy to visualize when $\ini(I)$ is the edge ideal of a graph $H$, the result is more effective when $H$ is a hypergraph, see Examples~\ref{tetrahedron} and ~\ref{hypergraph example}. In the case of a graph, the process is equivalent to packing \emph{stars} in $H$, where a star in $H$ is a subgraph of $H$ consisting of the closed neighborhood of a vertex $x$ and all edges containing $x$. 

The organization of the paper is as follows. In Section~\ref{sec.ini}  we identify situations where initially regular sequences are also regular sequences, see Theorem~\ref{regular equiv ini regular} and Corollary~\ref{regular sums of two}. In Section~\ref{sec.construction} we analyze properties of Gr\"{o}bner bases in a series of lemmas that we use to prove  the main result of our article, Theorem~\ref{main}. In Algorithm~\ref{alg}, we give a combinatorial interpretation of Theorem~\ref{main}, providing a way to construct initially regular sequences. We end Section~\ref{sec.construction} with some examples that illustrate how our results compare to previously known bounds, see Examples~\ref{higher power example}, \ref{better than epsilon}, and \ref{better than epsilon and tau, hypergraph}.

We devote Section~\ref{extensions} to finding  extensions of Theorem~\ref{main}. We show that certain sums of variables corresponding to disjoint pairs of leaves also yield regular sequences, see Theorem~\ref{leaf pairs regular}. We discuss when regular sequences and initially regular sequences can be combined to give longer initially regular sequences and therefore better estimates for the depth, see Proposition~\ref{append a seq}, Theorem~\ref{iterate repeated neighbors}, and Theorem~\ref{generalized iterate}. 
In addition, these extensions allow us to show that our depth bound is sufficiently robust to be used with polarizations, see Theorem~\ref{polarization}. Finally, in Section~\ref{applications} we give further applications of our results. In particular, we return to the case of non-monomial ideals and apply our techniques to ideals arising  from different combinatorial settings.

\begin{acknowledgement} The first author was partially supported by a grant from the Simons Foundation (grant \#244930). The second named author is partially supported by Simons Foundation (grant \#279786) and Louisiana Board of Regents (grant \#LEQSF(2017-19)-ENH-TR-25).
\end{acknowledgement}

\section{Regular and initially regular sequences} \label{sec.ini}

Throughout the paper, $R = k[x_1, \ldots, x_n]$ is a polynomial ring over an arbitrary field $k$. In this section, we show that the length of an initially regular sequence on $R/I$ gives a lower bound for $\depth R/I$, and discuss special situations where initially regular sequences are also regular sequences.

The following is a more general version of Definition~\ref{ini reg simple}. In many cases, a single fixed term order will be used to create an initially regular sequence. If only one term order is specified, it will be understood that all term orders used are the same and that all initial ideals are formed with respect to the fixed term order. More generally, one can take different term orders at every step of the construction of an initially regular sequence. For unexplained terminology, we refer the reader to \cite{AL}, \cite{BH}, and \cite{HH}.

\begin{definition} \label{ini reg}
Let $R$ be a polynomial ring over a field, and fix a set of term orders $>_1 , \ldots , >_{q}$ on $R$. Let $I \subseteq R$ be a proper ideal. A sequence of nonconstant polynomials $f_1, \ldots, f_q$ is said to be an \emph{initially regular sequence} on $R/I$ if for each $i = 1, \ldots, q$, $f_i$ is a regular element on $R/I_i$, where $I_i = \ini_{>_i} (I_{i-1}, f_{i-1})$ (here, by convention, $I_1 = \ini_{>_1}(I)$).
\end{definition}

In general, there is no relationship between an element being regular and initially regular. Given an ideal in a polynomial ring, an element can be both regular and initially regular, either one without the other, or neither. However, our first result shows that initially regular sequences give a lower bound on the depth.

\begin{proposition}\label{ini reg bound}
Let $I$ be an ideal in a polynomial ring $R$. If $f_1, \ldots, f_q$ form an initially regular sequence on $R/I$ with respect to a sequence of term orders $>_1, \ldots , >_q$, then
$$\depth R/I \geq q.$$
\end{proposition}

\begin{proof}
We proceed by induction on $q$. If $q=1$, then $f_1$ is regular on $R/\ini_{>_1}(I)$, and so by \cite[Theorem 3.3.4]{HH}, we have $\depth R/I \geq \depth R/\ini_{>_1}(I) \geq 1$.

Suppose that $q \geq 2$. Then $f_2, \ldots, f_q$ is an initially regular sequence on $R/(\ini_{>_1}(I),f_1)$ and by induction, $\depth R/(\ini_{>_1}(I),f_1) \geq q-1$.  Thus, by \cite[Theorem 3.3.4]{HH} again, we have $\depth R/I \geq \depth R/\ini_{>_1}(I)=\depth R/(\ini_{>_1}(I),f_1)+1 \geq q$.
\end{proof}

In light of \cite[Theorem 3.3.4]{HH}, which shows $\depth R/I \geq \depth R/\ini(I)$ for whichever term order is selected for the first step, we will often simplify statements by assuming step one has been completed. That is, when convenient, we can assume that we are starting with a monomial ideal.

In practice it is often the case that the initially regular sequence that we construct is also a regular sequence. We will show instances where the two notions are equivalent. To do so we first examine the initial ideal of $(I,b_0+b_1)$, where $b_0, b_1$ are distinct variables in a polynomial ring $R$ and $I$ is a monomial ideal in $R$. Note that in our examination, we describe a  set of monomials that generate $\ini(I,b_0+b_1)$. Although the set need not be minimal, it is convenient to describe the set in terms of the minimal generators of $I$.

To state the result we introduce a notation for the degree of a variable in a monomial. For a monomial $N$ and a variable $x$, define $d_x(N)= \max\left\{ t \mid x^t \text{ divides } N\right\}$.

\begin{lemma} \label{gens H}
Let $I$ be a monomial ideal in a polynomial ring $R$. Suppose that $b_0, b_1$ are distinct variables of $R$.
Fix a term order with $b_0>b_1$. Then
\begin{eqnarray*}
\ini (I,b_0+b_1)=\left(b_0, b_1^{d_{b_0}(M_j)} \frac{M_j}{b_0^{d_{b_0}(M_j)}} ~\Big|~ M_j \text{ a minimal monomial generator of } I\right).
\end{eqnarray*}
\end{lemma}

\begin{proof}
Let $R=k[x_1, \ldots, x_{r},b_0,b_1]$, $R_1=k[x_1, \ldots, x_r]$, and $R_2=k[b_0,b_1]$. We may write $I=(M_1, \ldots, M_{r_1}, M_{r_1+1}, \ldots M_p)$, where $b_0 \mid M_{j}$ for $1\leq j\leq r_1$ , and $b_0 \nmid M_j$ for $ r_1+1\leq j \leq p$. Observe that $b_0 \in \ini(I,b_0+b_1)$ since $b_0$ is the leading term of $b_0+b_1$.

In order to compute a Gr\"{o}bner basis for $(I,b_0+b_1)$ we must consider the reductions of all possible $S$-resultants. Notice that the $S$-resultant of two monomials is $0$. Thus, it initially suffices to consider all possible $S$-resultants involving $f_1=b_0+b_1$. Let $M$ be a monomial. Then $S(M,f_1)$ is again monomial. In fact,
\begin{eqnarray*}
S(M, f_1)=\frac{\lcm(M, b_0)}{b_0}(b_0+b_1)-\frac{\lcm(M,b_0)}{M}M=\frac{\lcm(M, b_0)}{b_0}b_1.
\end{eqnarray*}

 For $r_1+1\leq j \leq p$,  we have $$S(M_j, f_1)=\frac{\lcm(M_j, b_0)}{b_0}b_1=M_jb_1,$$ since $b_0 \nmid M_j$. In this case, $S(M_j,f_1)$ reduces to $0$ modulo the generators of $I$.

For $1\leq j \leq r_1$, we have $S(M_j, f_1)=\frac{\lcm(M_j, b_0)}{b_0}b_1=\frac{M_jb_1}{b_0}$.  If $b_0$ does not divide $\frac{M_jb_1}{b_0}$, then
$$\frac{M_jb_1}{b_0}=b_1^{d_{b_0}(M_j)} \frac{M_j}{b_0^{d_{b_0}(M_j)}}.$$
If $b_0$ divides $\frac{M_jb_1}{b_0}$ then we reduce by $b_0+b_1$ to get
$$\frac{M_jb_1}{b_0} -\left(\frac{M_jb_1}{b_0^2}\right)(b_0+b_1)=- \frac{M_jb_1^2}{b_0^2}.$$
Iterating the reduction step of the algorithm, we continue to reduce until the result is no longer divisible by $b_0$. That is, until we reach $ b_1^{d_{b_0}(M_j)} \frac{M_j}{b_0^{d_{b_0}(M_j)}}$. The assertion follows.
 \end{proof}

If $a,b$ is an initially regular sequence on $R/I$ and $f$ is an initially regular element on $R/(I,a,b)$, then $a,b,f$ need not be an initially regular sequence on $R/I$ since in general $\ini(I,a,b)\not=\ini(\ini(I,a),b)$ even in the case where only one term order is used. The situation is more clear when $a$ and $b$ are sums of two variables. In this case, Lemma~\ref{gens H} can be applied to obtain the following result.

\begin{lemma}\label{lem.initial}
Let $I$ be a monomial ideal in a polynomial ring $R$. Let $\{x_i,y_i\}_{i=1}^\ell$ be pairs of variables, and $>$ be a term order such that $x_i > y_i$ for all $i$ and $x_i > x_j$ for all $i < j$. Then
$$\ini(I,x_1+y_1, \ldots, x_\ell + y_\ell) = \ini(\ini(\ldots \ini(\ini(I,x_1+y_1),x_2+y_2), \ldots), x_\ell+y_\ell).$$
\end{lemma}

\begin{proof} By iterated use of Lemma~\ref{gens H}, the right hand side can be generated by $x_1, \ldots, x_\ell$ and the monomials obtained from the monomial generators of $I$ by successively replacing $x_1$ by $y_1$, then $x_2$ by $y_2$, \ldots, and eventually $x_\ell$ by $y_\ell$. Notice that once $x_i$ is replaced by $y_i$ in a monomial generator of $I$, $x_i$ will not reappear in the generating set through subsequent replacements since by the term order $x_i \neq x_j$ and $y_j \neq x_i$ for any $i < j$.

On the other hand, the left hand side is generated by the leading terms of a Gr\"obner basis of $(I, x_1+y_1, \ldots, x_\ell+y_\ell)$. Observe that the $S$-resultant of $x_i+y_i$ and $x_j+y_j$ reduces to $0$ modulo $\{x_1+y_1, \ldots, x_\ell+y_\ell\}$. As in Lemma~\ref{gens H}, the reduction of an $S$-resultant of a monomial with $x_i+y_i$ is formed by successively replacing $x_i$ by $y_i$ until $x_i$ no longer divides the resulting monomial. If this resulting monomial is divisible by $x_j$ for some $j$, the reduction process continues, eventually yielding a monomial where $x_j$ has been replaced by $y_j$ for all $j$. Hence, the left hand side can also be generated by the set consisting of $x_1, \ldots, x_\ell$ and the monomials obtained from monomial generators of $I$ by successively replacing $x_i$ by $y_i$, for $i = 1, \ldots, \ell$.
\end{proof}

It is interesting to note that not all the variables used in Lemma~\ref{lem.initial} need to be distinct. The conditions on the term order imply that $x_1, \ldots , x_{\ell}$ are distinct and that $x_i\not=y_j$ if $i\leq j$. However, $y_1, \ldots , y_{\ell}$ need not be distinct and $x_i=y_j$ for $i>j$ is possible. For example, the sets $\{x,y\}, \{z,y\}, \{y,w\}$ meet the conditions of Lemma~\ref{lem.initial} if $x>z>y>w$.

Our next result establishes an instance where the notion of initially regular is equivalent to being  regular.

\begin{theorem}\label{regular equiv ini regular}
Let $I$ be a monomial ideal in a polynomial ring $R=k[x_1, \ldots, x_r, b_0,b_1]$. Fix a term order with $b_0 > b_1$. Let $f \in R'=k[x_1,\ldots, x_r,b_1]$ be a polynomial. Then $f$ is regular on $R/(I,b_0+b_1)$ if and only if $f$ is initially regular on $R/(I, b_0+b_1)$.
\end{theorem}

\begin{proof}
Since $I$ is a monomial ideal, write $I=(M_1, \ldots, M_{\ell}, M_{\ell +1}, \ldots M_p)$, where $M_j$ is divisible by $b_0$ or $b_1$ if and only if $1\leq j \leq \ell$. For simplicity of notation, for a monomial $M$, set
$$\widehat{M} = b_1^{d_{b_0}(M)} \dfrac{M}{b_0^{d_{b_0}(M)}}.$$
By Lemma~\ref{gens H} we have
$$\ini(I, b_0+b_1)=\left(b_0, \widehat{M_1}, \ldots , \widehat{M_{\ell}}, M_{\ell +1}, \ldots , M_p \right).$$
Define $\phi: R \rightarrow R'$ to be the ring homomorphism that sends $b_0$ to $-b_1$ and identifies all other variables in $R$. It is easy to see that $\phi$ is onto and its kernel is $(b_0+b_1)$. Set
$$I'=\langle \ini(I,b_0+b_1)\setminus\{b_0\}\rangle = \left(\widehat{M_1}, \ldots , \widehat{M_{\ell}}, M_{\ell +1}, \ldots , M_p \right).$$
Define $\overline{\phi}:R \rightarrow R'/I'$ by $\overline{\phi}(a)=\phi (a) +I'$ for $a\in R$. Notice that $\overline{\phi}$ is an onto homomorphism, and since $\phi(I) \subseteq I'$, $(I, b_0+b_1) \subseteq {\mbox{\rm ker}}(\overline{\phi})$.
To see that this is an equality, consider an arbitrary monomial $M$ of $R$ such that $\phi(M) \in I'$. If $M$ is not divisible by $b_0$ or $b_1$, then $\phi(M) = M$ and $M \in \langle M_{\ell +1}, \ldots, M_p \rangle \subseteq I$. If $b_0$ or $b_1$ divides $M$, then $b_1$ divides $\phi(M)$ and thus $\phi(M)$ is divisible by $\widehat{M_j}$ for some $1\leq j \leq \ell$. By the definition of $\phi$, this implies that there is a monomial $N=b_0^{t_0}b_1^{t_1}{\frac{\widehat{M_j}}{b_1^{t}}}$ that divides $M$ for some $t_0+t_1= t=d_{b_1}(\widehat{M_j})$.
Set $M_j'= {\frac{\widehat{M_j}}{b_1^{t}}}$.  Notice that $M_j=b_0^{s_0}b_1^{s_1}M_j'$ for some $s_0,s_1 \geq 0$ with $s_0+s_1=t$.

Next we claim that $N \in (I,b_0+b_1)$. If $t_0=s_0$, then $t_1=s_1$ and $N=M_j \in I \subseteq (I,b_0+b_1)$. If $s_0 \geq 1$, then ${\frac{M_jb_1}{b_0}}=(b_0+b_1){\frac{M_j}{b_0}} - M_j  \in (I,b_0+b_1)$. Iterating the process shows that $N \in (I, b_0+b_1)$ for all $t_0\leq s_0$. Similarly, if $s_1 \geq 1$, then ${\frac{M_jb_0}{b_1}}=(b_0+b_1){\frac{M_j}{b_1}} - M_j  \in (I,b_0+b_1)$. Again, iterating the process shows that $N \in (I, b_0+b_1)$ for all $t_1\leq s_1$. Since $t_0+t_1=s_0+s_1$, this shows $N\in (I,b_0+b_1)$ for all such $N$.
Thus, ker$(\overline{\phi})=(I,b_0+b_1)$ and so
$$R/(I,b_0+b_1) \cong R'/I'.$$

Notice that $R/\ini(I, b_0+b_1) \cong R/(I', b_0) \cong R'/I' \cong R/(I, b_0+b_1)$ and therefore, since $\phi(f)=f$, then $f$ is regular on $R/(I,b_0+b_1)$ if and only if $f$ is regular on $R/\ini(I, b_0+b_1)$.
\end{proof}

In the setting of Theorem~\ref{regular equiv ini regular}, suppose $f \in R''=k[x_1, \ldots, x_r]$. Then it follows directly from the proof that the roles of $b_0$ and $b_1$ can be reversed.

\begin{corollary}\label{regular sums of two}
Let $I$ be a monomial ideal in a polynomial ring $R$.
Let $C=\{c_1, \ldots, c_r\}$ be a set of distinct variables, and let $B=\{x_1, \ldots, x_{q-1}, y_1, \ldots, ,y_{q-1}\}$ be a collection of variables disjoint from $C$ with a possible exception of $c_1=y_{q-1}$. Fix a term order $>$ such that $x_i > y_i$ for all $i$, $x_i > x_j$ for all $i < j$, and $c_1>c_i$ for all $2\leq i \leq r$.
Set $f_i=x_i+y_i$
 for all $1\leq i \leq q-1$ and let $f_{q}=c_1+\ldots +c_r$.
 Then,
 $f_1, \ldots, f_{q}$ is a regular sequence on $R/I$ if and only if $f_1, \ldots, f_q$ is an initially regular sequence on $R/I$.
 \end{corollary}

\begin{proof}
The result follows by induction, using Theorem~\ref{regular equiv ini regular} and Lemma~\ref{lem.initial}.
\end{proof}

Note that the condition that the ideal $I$ is monomial in Corollary~\ref{regular sums of two} is necessary, so for a general ideal $I$, the corollary yields a sequence that is regular on $R/\ini(I)$. We shall apply Corollary~\ref{regular sums of two} later on, in Section~\ref{extensions}, to construct examples of initially regular sequences which are also regular sequences.

\section{Constructing initially regular sequences} \label{sec.construction}

This section is devoted to the task of deriving an algorithm to construct initially regular sequences. Recall that Proposition~\ref{ini reg bound} gives a lower bound for the depth of $R/I$ provided that initially regular sequences on $R/I$ can be found.

Since initially regular sequences require that an element is regular on an initial ideal at each step, it is helpful to understand the structure of the Gr\"{o}bner basis at each step in the construction. For background information on Gr\"obner bases or Buchberger's algorithm, see \cite{AL}. For simplicity and convenience of notation, in the remainder of this section, unless otherwise specified, we shall assume the following set-up:

\begin{setup}\label{setup1}
Let $R = k[x_1, \ldots, x_r, b_0, \ldots, b_t]$ be a polynomial ring over a field and let $I$ be a monomial ideal in $R$. Suppose that $R$ has a fixed term order, and set $R_1 = k[x_1, \ldots, x_r]$ and $R_2 = k[b_0, \ldots, b_t]$.

\end{setup}

The following notion allows us to focus on ideals and Gr\"obner bases of a particular form, which is an essential part of our construction of initially regular sequences.

\begin{definition}
A polynomial $f \in R$ is called {\it{$(R_1, R_2)$-factorable}} if $f=Mg \in R$, where $M \in R_1$ is a monomial and $g\in R_2$ is a polynomial. An ideal $I$ that admits a minimal set of generators $\{f_1, \ldots, f_p\}$ in which $f_i$ is $(R_1, R_2)$-factorable for all $i$ is called an {\it{$(R_1, R_2)$-factorable} ideal}. A Gr\"obner basis whose elements are $(R_1, R_2)$-factorable is called an {\it{$(R_1, R_2)$-factorable Gr\"{o}bner basis}}.
\end{definition}

Note that all monomial ideals are $(R_1, R_2)$-factorable, where $1$ is considered a monomial in $R_1$ when necessary.
The next two lemmas show that the two key steps of Buchberger's algorithm, forming $S$-resultants and the reduction process, preserve $(R_1, R_2)$-factorability. Recall that for a polynomial $g$ in $R$ under a fixed term order, $\ini(g)$ represents the leading term of $g$.

\begin{lemma}\label{S-resultant}
Let $f,g\in R$ be $(R_1, R_2)$-factorable polynomials.Then $S(f,g)$ is also $(R_1,R_2)$-factorable.
\end{lemma}

\begin{proof}
By assumption, there exist monomials $M,N \in R_1$ and polynomials $f', g' \in R_2$ with $f=Mf'$ and $g=Ng'$. Write $f'=f_1+\tilde{f}$ and $g'=g_1+\tilde{g}$, where $f_1$ and $g_1$ are the leading terms of $f'$ and $g'$ respectively under the fixed term order. By the definition of an $S$-resultant, and using the fact that $M$ and $N$ are monomials in $R_1$, we have
\begin{eqnarray*}
S(Mf',Ng') & = & \frac{\lcm (Mf_1,Ng_1)}{Mf_1}Mf' - \frac{\lcm(Mf_1,Ng_1)}{Ng_1}Ng' \\
&=& \frac{\lcm(M,N)\lcm(f_1,g_1)}{f_1}f'-\frac{\lcm(M,N)\lcm(f_1,g_1)}{g_1}g'\\
&=& \lcm(M,N)S(f',g').
\end{eqnarray*}
\end{proof}

\begin{lemma}\label{reduction step}
Let $f,g\in R$ be $(R_1, R_2)$-factorable polynomials. When $f$ is reduced modulo $g$ in Buchberger's algorithm, then the remainder will also be $(R_1, R_2)$-factorable. Moreover, the $R_1$-monomial term in the $(R_1, R_2)$-factorization of $f$ is the same as the $R_1$-monomial term of the remainder.
\end{lemma}

\begin{proof} Suppose that $f=Mf', g=Ng'$, where  $M, N \in R_1$ are monomials, and $f',g' \in R_2$ are polynomials.  Write $f = \sum_{i=1}^{t_1}f_i=M\sum_{i=1}^{t_1}f_i'$ and $g=\sum_{j=1}^{t_2}g_j=N\sum_{j=1}^{t_2}g_j'$, where the $f_i,g_j$ are the monomial terms of $f$ and $g$, respectively.

Observe that $f$ can be reduced modulo $g$ in the Buchberger's algorithm  if the leading term $g_1 = \ini(g)$ divides a monomial term $f_i$. That is, $f_i=\alpha g_1$ for some monomial $\alpha \in R$. Write $f'=f_i'+\tilde{f}$ and $g'=g_1'+\tilde{g}$ for some polynomials $\tilde{f}, \tilde{g}\in R_2$.

Let $h=f-\alpha g$.  Write $\alpha = \alpha_1 \alpha_2$, where $\alpha_1 \in R_1$ and $\alpha_2 \in R_2$. Since $Mf_i'=\alpha Ng_1'$, it follows that  $\alpha_1 N = M$, and $\alpha_2 g_1'=f_i'$. Therefore,
\begin{eqnarray*}
h & = & f-\alpha g = Mf'- \alpha N g' \\
 & = & Mf_i'+ M \tilde{f} - \alpha N g_1' - \alpha N  \tilde{g} \\
 & = & M \tilde{f} - \alpha_1N \alpha_2 \tilde{g} = M(\tilde{f} - \alpha_2 \tilde{g}).
\end{eqnarray*}
The conclusion now follows since the remainder of $f$ modulo $g$ is obtained by repeating this process until no monomial term of $f$ is divisible by $\ini(g)$.
\end{proof}

Using Lemmas~\ref{S-resultant} and \ref{reduction step}, we show that Buchberger's algorithm preserves $(R_1,R_2)$-factorability.

\begin{proposition}\label{GrobnerArgument}
Let $I=(f_1, \ldots , f_p)$ be an $(R_1, R_2)$-factorable ideal in $R$ such that for each $i = 1, \ldots, p$, $f_i=M_ig_i$, where $M_i \in R_1, g_i \in R_2$, and $M_i$ are monomials. Then, there exists an  $(R_1, R_2)$-factorable Gr\"{o}bner basis of $I$ in which every element is of the form $f=Mg$, where $M \in R_1, g \in R_2$, and $M=\lcm(M_{i_{1}}, \ldots, M_{i_{\ell}})$, for some $1 \leq i_1, \ldots, i_{\ell} \leq p$. Furthermore, the unique reduced Gr\"{o}bner basis of $I$ is also $(R_1, R_2)$-factorable and consists of elements of this form.
\end{proposition}

\begin{proof}
We follow Buchberger's algorithm to produce a Gr\"{o}bner basis for $I$. Set $G_1=\{f_1, \ldots, f_p\}$. For $i \not= j$ form the S-resultant $S=S(f_i,f_j)$. By Lemma~\ref{S-resultant}, $S$ has the desired form with $S=\lcm(M_i, M_j)S(g_i,g_j)$. If $\ini (f_k)$ divides $\ini(S)$ for any $k$, reduce $S$ modulo $f_k$. Note that by Lemma~\ref{reduction step} the reduction has the desired form and the monomial term of the reduction remains $\lcm(M_i, M_j)$. Repeat this process until $S = \sum_{k=1}^p \alpha_k f_k + f_{p+1}$, where $\ini(f_k)$ does not divide $\ini(f_{p+1})$ for all $k$. That is, $f_{p+1}$ is the remainder when $S$ is reduced modulo $G_1$.   If $f_{p+1} \not= 0$, add it to $G_1$.  Thus, the new set $G_1$ again consists entirely of $(R_1,R_2)$-factorable elements.
Repeating this process produces a Gr\"{o}bner basis $G_1=\left\{ f_1, \ldots, f_p,f_{p+1}, \ldots , f_n\right\}$, where every element has the desired form.

To produce the (unique) reduced Gr\"{o}bner basis, the elements of $G_1$ need to be further reduced so that for each $1\leq i \leq n$, $\ini(f_i)$ does not divide any monomial term of $f_j$ for $i\not= j$. Again by Lemma~\ref{reduction step}, passing to the reduced Gr\"obner basis preserves $(R_1,R_2)$-factorability and the form of the $R_1$-monomial terms.
\end{proof}

A closer examination of the proof in Proposition~\ref{GrobnerArgument} shows that if $I$ is an $(R_1, R_2)$-factorable  ideal, the maximum degree of a variable $x_i$ that divides one of the generators of $I$ will not increase when passing to an $(R_1, R_2)$-factorable  Gr\"{o}bner basis. In particular, if the monomial terms $M_i \in R_1$ associated to the original generating set of $I$ are squarefree, then so are the $R_1$-monomial terms of the Gr\"{o}bner basis. Recall that for a monomial ideal $I = (f_1, \ldots, f_p)$ we  set $d_x(I) = \max\{ d_x(f_i) \mid1 \leq i \leq p\}$. Notice that this is well defined as the set of minimal monomial generators of $I$ is unique.

\begin{corollary}\label{degrees}
Let $I=(f_1, \ldots , f_p)$ be an $(R_1, R_2)$-factorable  ideal with $f_i=M_ig_i$. Then $\max \left\{ d_x(M_i)\right\} \geq d_{x}(\ini(I))$ for every variable $x$ in $R_1$.
\end{corollary}

\begin{proof}
Let  $G=\{h_1, \ldots , h_m\}$ be the reduced Gr\"{o}bner basis for $\ini(I)$ and notice that by Proposition~\ref{GrobnerArgument}, $G$ is an $(R_1, R_2)$-factorable  Gr\"{o}bner basis.
Let  $h_i = N_ih_i'$ for $i= 1, \ldots, m$ with $N_j$ monomials in $R_1$ and $h_j' \in R_2$.

By Proposition~\ref{GrobnerArgument} , each $N_i$ is the least common multiple of some of the $M_1, \ldots, M_p$. Thus, the assertion follows by observing that
\begin{align*}
d_x(\lcm(M_{i_1}, \ldots, M_{i_{\ell}})) & = \max \{d_x(M_{i_j}) \mid 1 \leq j \leq {\ell}\} \\
& \leq \max \{d_x(M_i) \mid 1 \leq i \leq p\}.
\end{align*}
\end{proof}

Additional control over the maximal degree of a variable will be needed in special cases once we begin to form the initially regular sequences. The following lemma provides such control. Recall that we are still in the setting of Set-up~\ref{setup1}.

\begin{lemma}\label{degreesFixed}
Let $I$ be a monomial ideal in $R$. Let $J=(I, b_0+b_1+\ldots + b_t)$. Then $d_{x_i}(I) \geq d_{x_i}(\ini(J))$ for all $1 \leq i \leq r$. Furthermore, suppose that $>$ is a  monomial order such that $b_0 > b_j$ for all $j = 1, \ldots, t$. Then $d_{b_0}(\ini(J)) = 1$.
\end{lemma}

\begin{proof} Since all monomial ideals are $(R_1, R_2)$-factorable, let $I = (f_1, \ldots, f_p)$, where $f_i=M_ig_i$ with $M_i \in R_1$, $g_i \in R_2$, and $M_i,g_i$ monomials. Setting $M_{p+1}=1$ we have $f_{p+1} = b_0 + b_1 + \ldots + b_t = M_{p+1} f_{p+1}$ and thus $J$ is an $(R_1, R_2)$-factorable  ideal. Thus, by Proposition~\ref{GrobnerArgument}, there exists an $(R_1, R_2)$-factorable Gr\"{o}bner basis of $J$. The first statement follows from Corollary~\ref{degrees} after noting that for every variable $x_i$,
$$d_{x_i}(I) = \max \{ d_{x_i}(f_\ell) \mid 1 \leq \ell \leq p+1\}= \max \{ d_{x_i}(M_\ell) \mid 1 \leq \ell \leq p+1\}   .$$
The second statement is obvious since the leading term of $b_0 + \ldots + b_t$ is $b_0$.
\end{proof}

The next result shows that when $I$ is a monomial ideal and we form a colon ideal with a sum of variables, then under some mild conditions on the degrees of the variables, the resulting ideal is still a monomial ideal.

\begin{lemma}\label{r=1 case}
Let $I$ be a monomial ideal in a polynomial ring $R$. Suppose that $h=\sum \limits_{i=0}^{t}b_i$ is a sum of distinct variables in $R$ and suppose that $d_{b_i}(I) \leq 1$ for all $i \geq 1$. Then $I:h$ is a monomial ideal.
\end{lemma}

\begin{proof}
Suppose that $f \in R$ and $fh \in I$. Since $I$ is monomial and $h$ is homogeneous, we may assume that $f$ is a homogeneous polynomial.  Let $f = \sum \limits_{i = 1}^{\ell} f_i=\sum \limits_{i=1}^{\ell}c_if_i'$, where $f_1', \ldots, f_{\ell}'$ are distinct monomials of the same degree and $f_i=c_if_i'$ with $c_i \in k$  for all $i$. If $f_i(b_0+\ldots+b_t) \in I$ for some $1\leq i \leq \ell$ and $\ell \geq 2$, then we may replace $f$ by $f - f_i$. Thus, we can assume that either $\ell =1$ or $\ell \geq 2$ and $f_ih \not\in I$ for all $i = 1, \ldots, \ell$. It suffices to show that $\ell=1$.

Suppose that $\ell\geq 2$ and for every $1\leq i\leq \ell$ we have $f_ih=f_i(b_0+\ldots +b_t) \not\in I$. Since $I$ is a monomial ideal, this implies that for each $1\leq i \leq \ell$, we have $f_ib_j \not\in I$ for some $ 0\leq j \leq t$. Observe that among those pairs $(i,j)$ such that $f_ib_j \not\in I$, there must be such a pair in which $j \not= 0$. Indeed, if that is not the case then we would have $f_ib_j \in I$ for all $i$ and $j > 0$ and hence $f_ib_0 \not\in I$ for all $i$. This would imply that $fb_0 = f(b_0 + \ldots + b_t) - f(b_1 + \ldots +b_t) \in I$, a contradiction.

Now, among all pairs $(i,j)$ (with $j > 0$) such that $f_ib_j \not\in I$, let $(\alpha,\beta)$ be such a pair so that $d_{b_\beta}(f_\alpha)$ is maximal possible. Since $f_\alpha b_\beta \not\in I$, then as $f(b_0 + \ldots + b_t) \in I$, there must exist $(\gamma_i,\delta_i)$, for $1 \leq i \leq v$, with $\delta_i \not= \beta$ such that $f_\alpha b_\beta + \sum\limits_{i=1}^v f_{\gamma_i} b_{\delta_i}=0$. That is, $c_{\alpha}f'_\alpha b_\beta + \sum\limits_{i=1}^v c_{\gamma_{i}}f'_{\gamma_i} b_{\delta_i}=0$, where $c_{\alpha}+\sum\limits_{i=1}^v c_{\gamma_{i}} = 0$ and $f'_\alpha b_\beta = f'_{\gamma_i} b_{\delta_i}$ for each $i$. In particular, set $f_\gamma = f_{\gamma_{1}}$ and $b_{\delta}=b_{\delta_{1}}$. Then $f_{\gamma} b_\delta \not\in I$.

Observe further that $d_{b_\beta}(f_\gamma) = d_{b_\beta}(f_\gamma b_\delta) = d_{b_\beta}(f_\alpha b_\beta) = d_{b_\beta}(f_\alpha) + 1$. Thus, by the maximality of $d_{b_\beta}(f_\alpha)$, we must have $f_\gamma b_\beta \in I$. However, $d_{b_\beta}(f_\gamma b_\beta) = d_{b_\beta}(f_\gamma) + 1 \geq 2$, and since $d_{b_\beta}(I)\leq 1$ we must have that  $f_\gamma \in I$, which is a contradiction.
\end{proof}

We are now ready to begin creating an initially regular sequence on $R/I$.

\begin{lemma}\label{regular}
Suppose $I$ is a monomial ideal in a polynomial ring $R$. Let $b_0, b_1, \ldots, b_t$ be distinct variables of $R$ such that:
\begin{enumerate}[$($a$)$]
\item $d_{b_i}(I) \leq 1$, for all $i \geq 1$; and
\item if $M$ is a monomial generator of $I$ and $b_0$ divides $M$, then there exists an $i \geq 1$ such that $b_i$ divides $M$.
\end{enumerate}
Then $b_0 + b_1 + \ldots + b_t$ is a regular element on $R/I$.
\end{lemma}

\begin{proof}
Suppose that $f \in R$ and $f(b_0 + \ldots + b_t) \in I$. We shall show that $f \in I$. Since $I$ is a monomial ideal, we may assume that $f$ is a monomial by Lemma~\ref{r=1 case}.
Then
$f(b_0 + \ldots + b_t) \in I$ implies that $fb_j \in I$ for every $j \geq 0$. Since $fb_0 \in I$, we can write $fb_0 = Mg$ for a minimal generator $M$ of $I$. It follows that either $b_0 \mid g$ in which case  $f \in I$, or $b_0 \mid M$. If $b_0 \mid M$, condition (b) implies that $b_i$ divides $M$ for some $i \geq 1$. This, in particular, shows that $b_i \mid f$. Now, consider $fb_i \in I$. Noting that $d_{b_i}(fb_i) \geq 2$, condition (a) then implies that $f \in I$.
\end{proof}

\begin{remark}\label{degree 1 relaxed}
Notice that condition $($a$)$ in Lemma~\ref{regular} requires that all the $b_i$ have degree at most one for all $i\geq 1$. This allows us to consider polynomials where $b_0$ has a higher degree. A careful examination of the proofs of Lemmas~\ref{r=1 case} and \ref{regular} shows that the condition $d_{b_i}(I) \leq 1$ for all $i \geq 1$ can be relaxed to instead require that the degree of $b_i$ in all monomial generators of $I$ is either $0$ or a fixed constant.
\end{remark}

We are ready to state our primary theorem, which provides a process for creating an initially regular sequence for any ideal in a polynomial ring. Note that in this first version, we are giving the basics. There are special cases where we can fine-tune the process to achieve improved lower bounds on the depth. These cases will be discussed in detail in Section~\ref{extensions}.

\begin{theorem} \label{main}
Let $I$ be an ideal in a polynomial ring $R$ and $>_1$ a term order. Suppose that $\{b_{i,j} \mid 1 \leq i \leq q, 0 \leq j \leq t_i\}$ are distinct variables of $R$ such that:
\begin{enumerate}[$($1$)$]
\item $d_{b_{i,j}}(\ini_{>_1}(I)) \leq 1$, for all $i \geq 1$ and $j \geq 1$; and
\item for each $i = 1, \ldots, q$, if $M$ is a monomial generator of $\ini_{>_1}(I)$ and $b_{i,0}$ divides $M$, then there exists a $j \geq 1$ such that $b_{i,j}$ divides $M$.
\end{enumerate}
Let $f_i = \sum_{j=0}^{t_i} b_{i,j}$, for $1 \leq i \leq q$. Then $f_1, \ldots, f_q$ is an initially regular sequence on $R/I$ with respect to $>_1$ and any term orders $>_2, \ldots , >_q$ for which $b_{i,0} >_{i+1} b_{i,j}$ for all $1\leq i \leq q-1$ and all $j$. In particular, $\depth R/I \geq q$.
\end{theorem}

\begin{proof} The last statement follows from Proposition~\ref{ini reg bound}. By replacing $I$ with $\ini_{>_1}(I)$ we may assume first that $I$ is a monomial ideal.  We proceed by induction on $q$. For $q=1$, the result follows from Lemma~\ref{regular}. Suppose that $q \geq 2$.

By Lemma~\ref{regular}, $f_1$ is regular on $I$. Let $J = \ini_{>_2}(I, f_1)$. It suffices to show that $f_2, \ldots, f_q$ form an initially regular sequence on $R/J$ with respect to the remaining term orders. Indeed, by the induction hypothesis, it is enough to show that $\{b_{i,j} \mid 2 \leq i \leq q, 0 \leq j \leq t_i\}$ satisfy conditions (1) and (2) of the hypotheses relative to $J$.

By Lemma~\ref{degreesFixed} (by letting $R_2 = k[b_{1,0}, \ldots, b_{1,t_1}]$), we first have $d_{b_{i,j}}(J) \leq d_{b_{i,j}}(I) \leq 1$ for all $2 \leq i \leq q$ and $1 \leq j \leq t_i$.
Now, suppose that for some $2 \leq i \leq q$ we have $b_{i,0} \mid M$, for some minimal monomial generator $M$ of $J$. By Proposition~\ref{GrobnerArgument}, $M$ can be written as $M = Ng$, where $N$ is the least common multiple of a collection $\{M_{i_1}, \ldots, M_{i_{\mu}}\}$ of $R_1$-monomial factors of minimal generators of $I$ and $g$ is in $R_2$. Since $b_{i,0}$ divides $M$, then $b_{i,0} \mid N$, since $b_{i,0} \not\in R_2$.
Thus $b_{i,0}$ must divide $M_{i_\ell}$ for some $1 \leq \ell \leq \mu$. In particular, $b_{i,0}$ must divide a minimal generator $M_{i_{\ell}}g_{i_{\ell}}$ of $I$, where $g_{i_{\ell}} \in R_2$. By the hypothesis, this implies that there exists $1 \leq j \leq t_i$ such that $b_{i,j}$ divides $M_{i_\ell}g_{i_{\ell}}$, i.e., $b_{i,j}$ divides $M_{i_{\ell}}$.
The result now follows.
\end{proof}

Notice that Theorem~\ref{main simple} follows immediately by taking $>_i$ to be the same term order for all $ 1\leq i \leq q$. Theorem~\ref{main} leads us to an algorithm for constructing an initially regular sequence for any ideal in a polynomial ring based on the combinatorial data of an appropriate hypergraph.

\begin{algorithm} \label{alg}
Given an arbitrary ideal $I \subseteq R = k[x_1, \ldots, x_n]$:
\begin{itemize}
\item Step 1: Choose a term order in $R$ and compute $J = \ini(I)$.
\item Step 2: Let $H = (V_H, E_H)$ be the underlying hypergraph associated to $J$ (with degree at each vertex representing the highest power to which the vertex appears in the generators of $J$).
\item Step 3: Set $L = [ \quad ]$ be an empty list and $S=\emptyset$.
\item Step 4: Pick a vertex $b_0\in V_H\setminus S$.
\item Step 5: Let $\mathcal{E}=\{e \in E_H \mid b_0\in e\}$ be edges of $H$ containing $b_0$. Select a set of vertices $B=\{b_1, \ldots, b_t\}$ such that $b_i \not\in S$ for all $i$, $B\cap e \not= \emptyset$ for all $e \in \mathcal{E}$, and for all $i$, $b_i$ has degree at most $1$. (To optimize the process, select $B$ to be minimal with respect to inclusion).
\item Step 6: Let $f = b_0 + \sum_{i=1}^t b_i$. Append $f$ to $L$ and add $b_0, b_1, \ldots, b_t$ to $S$.
\item Step 7: Repeat Steps 4 -- 6 until either $S=V_H$ or for any $b_0\in V_H\setminus S$ there does not exist a set $B$ satisfying the conditions above.
\item Output: the list $L$, which forms an initially regular sequence on $R/I$ with respect to an appropriate term order.
\end{itemize}
\end{algorithm}

\begin{remark} There are freedoms of choice in Algorithm~\ref{alg} which, in practice, can be utilized to give us a sharper bound for the depth of $R/I$. Specifically,
\begin{enumerate}
\item the initial term order in $R$ can be chosen so that $J = \ini(I)$ and $H = H(J)$ are combinatorially easy to visualize, as was done in Example~\ref{extended pentagon}; and
\item (in high generating degrees) the variables $b_i$ in Step 5 can be chosen appropriately so that the iterated process can be done as many times as possible.
\end{enumerate}
\end{remark}

In condition~$($1$)$ of Theorem~\ref{main} we require the degree of all the variables we use to build the initially regular sequences to be one, with the exception of the degrees of the variables $b_{i,0}$. The following example illustrates this.

\begin{example} \label{higher power example}
Let $R=\mathbb{Q}[a,b,c,d]$ and let $I=(a^2b, abcd, c^2d)$ be the ideal corresponding to the hypergraph $G$ depicted below, where the degrees of the vertices $a$ and $c$ indicate the maximal power to which they appears in a minimal generator of $I$.

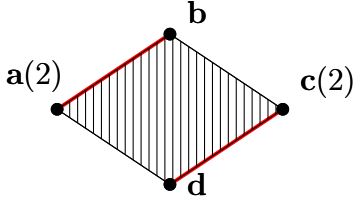
\begin{tikzpicture}

    \tikzstyle{point}=[circle,thick,draw=black,fill=black,inner sep=0pt,minimum width=4pt,minimum height=4pt]

 \node (a)[point,label={[xshift=-0.3cm, yshift=0 cm] $\bf{a}$(2)}] at (0.5,1) {};

    \node (b)[point,label={[label distance=0cm]3:$\bf{b}$}] at (2,2) {};

\node (c)[point,label={[label distance=0cm]3:$\bf{c}$(2)}] at (3.5,1) {};

\node (d)[point,label={[label distance=0cm]1:$\bf{d}$}] at (2,0) {};

 \draw[line width=0.5mm, red] (a.center) -- (b.center);
\draw(b.center)--(c.center);
 \draw (a.center) -- (d.center);
 \draw[line width=0.5mm, red] (c.center) -- (d.center);

 \node (a)[point,label={[xshift=-0.3cm, yshift=0 cm] $\bf{a}$(2)}] at (0.5,1) {};

    \node (b)[point,label={[label distance=0cm]3:$\bf{b}$}] at (2,2) {};

 \draw[pattern=vertical lines] (a.center) -- (b.center) -- (c.center) -- (d.center) -- cycle;

\node (c)[point,label={[label distance=0cm]3:$\bf{c}$(2)}] at (3.5,1) {};

\node (d)[point,label={[label distance=0cm]1:$\bf{d}$}] at (2,0) {};

\end{tikzpicture}

Notice that by Theorem~\ref{main} and Corollary~\ref{regular sums of two} we have $a+b, c+d$ is both a regular and an initially regular sequence on $R/I$ with respect to any term order such that $a>b$ and $c>d$. Using Macaulay~2~\cite{M2} we can confirm that $\depth R/I=2$.

\end{example}

In the following examples we apply Algorithm~\ref{alg} to obtain initially regular sequences and bounds on the depth of $R/I$. We also explain how the bounds obtained compare to known bounds.

\begin{example} \label{better than epsilon}
Let $I=(x_1x_2,x_2x_3, x_1x_3,x_2x_4,x_4x_5,x_3x_6,x_6x_7) \subseteq R=\mathbb{Q}[x_1,\ldots, x_7]$ be the edge ideal of the graph depicted below.

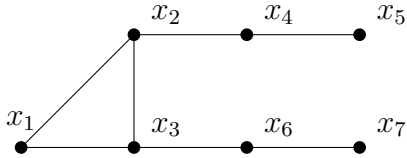
\setlength{\unitlength}{0.7cm}
\begin{tikzpicture}

    \tikzstyle{point}=[circle,thick,draw=black,fill=black,inner sep=0pt,minimum width=4pt,minimum height=4pt]

 \node (a)[point,label={[label distance=0cm]:$x_1$}] at (0.5,0.5) {};

    \node (b)[point,label={[label distance=0cm]3:$x_2$}] at (2,2) {};

\node (d)[point,label={[label distance=0cm]3:$x_4$}] at (3.5,2) {};

 \node (e)[point,label={[label distance=0cm]3:$x_5$}] at (5,2) {};

 \node (g)[point,label={[label distance=0cm]3:$x_7$}] at (5,0.5) {};

    \node(c)[point,label={[label distance=0cm]3:$x_3$}] at (2,0.5) {};

   \node(f)[point,label={[label distance=0cm]3:$x_6$}] at (3.5,0.5) {};

\draw (a.center) -- (b.center);
\draw (a.center) -- (c.center);
\draw(b.center)--(d.center);
\draw(d.center)--(e.center);
\draw(b.center)--(c.center);
\draw (c.center) -- (f.center);
\draw (f.center) -- (g.center);
\end{tikzpicture}

Previously known bounds from \cite{DS} give $\depth R/I \ge 3$. Theorem~\ref{main} confirms that $\depth R/I \ge 3$ in this example. Notice that $x_5+x_4, x_7+x_6, x_1+x_2+x_3$ is an initially regular sequence on $R/I$, where  $x_1>x_2>x_3>x_5>x_4>x_7>x_6$. Computations in Macaulay~2~\cite{M2} indeed verify that $\depth R/I=3$.  It is also worth noting that $ x_5+x_4, x_7+x_6, x_1+x_2+x_3$ is a regular sequence on $R/I$ as well by Corollary~\ref{regular sums of two}.
\end{example}

The next example shows that in the case of hypergraphs a careful selection of the vertex sets used can result in a significant improvement from known results.

\begin{example}\label{better than epsilon and tau, hypergraph}
Let $I=(x_1x_2x_3, x_2x_3x_4,x_2x_5x_6,x_3x_7x_8,x_4x_9x_{10}) \subseteq R=\mathbb{Q}[x_1,\ldots, x_{10}]$ be the edge ideal of the hypergraph  depicted below. 

\begin{tikzpicture}

    \tikzstyle{point}=[circle,thick,draw=black,fill=black,inner sep=0pt,minimum width=4pt,minimum height=4pt]

 \node (a)[point,label={[label distance=0cm]:$x_1$}] at (0.5,1.25) {};

    \node (b)[point,,label={[xshift=0.4cm, yshift=-0.3 cm]:$x_2$}] at (2,2) {};

\node (d)[point,,label={[xshift=0cm, yshift=0.1 cm]: $x_4$}] at (3.5,1.25) {};
  \node(c)[point,,label={[xshift=0.4cm, yshift=-0.4 cm]:$x_3$}] at (2,0.5) {};

   \node (e)[point,label={[label distance=0cm]:$x_5$}] at (1,3.25) {};

    \node (f)[point,label={[label distance=0cm]3:$x_6$}] at (3,3.25) {};

\node (g)[point,,label={[xshift=-0.4cm, yshift=-0.3 cm]: $x_7$}] at (1,-.75) {};
  \node(h)[point,,label={[xshift=0.4cm, yshift=-0.3 cm]:$x_8$}] at (3,-.75) {};
 \node (i)[point,,label={[xshift=0.4cm, yshift=-0.3 cm]:$x_9$}] at (5,2) {};

    \node (j)[point,,label={[xshift=0.4cm, yshift=-0.3 cm]:$x_{10}$}] at (5,.5) {};

   \draw[pattern=north east lines] (a.center) -- (c.center) -- (b.center) -- cycle;
   \draw[pattern=north west lines] (b.center) -- (c.center) -- (d.center) -- cycle;
\draw (a.center) -- (b.center);
\draw (a.center) -- (c.center);
\draw (b.center)-- (d.center);
\draw (b.center)-- (c.center);
\draw (c.center)-- (d.center);
\draw [pattern=north east lines] (b.center) -- (e.center) -- (f.center) -- (b.center);
\draw[pattern=north east lines] (c.center) -- (g.center) -- (h.center) -- (c.center);
\draw[pattern=north west lines] (d.center) -- (i.center) -- (j.center) -- (d.center);

\end{tikzpicture}

The previously known bound of \cite[Theorem~3.2]{DS2} shows $\depth R/I \ge 1$. The bounds from this section ensure $\depth R/I \geq 4$.
Notice that $x_1+x_2, x_5+x_6, x_7+x_8, x_9+x_{10}$ is both a regular and an initially regular sequence on $R/I$, where $x_1 > x_2 > x_5 >x_6 > x_7 > x_8 > x_3 > x_9 > x_{10} > x_4$  by Theorem~\ref{main} and Corollary~\ref{regular sums of two}.

Using Macaulay~2~\cite{M2} we have that $\depth R/I=6$. In the next section, we show that our results can be further refined to improve accuracy. For example, using Theorem~\ref{iterate repeated neighbors} we will see that $x_1+x_2, x_5+x_6, x_7+x_8, x_8+x_3, x_9+x_{10},x_{10}+x_4$ is both a regular and an initially regular sequence on $R/I$  (relative to the order above) and thus achieving the actual bound for $\depth R/I$.
\end{example}

\section{Extensions of initially regular sequences}\label{extensions}

In this section, we discuss some extensions of Theorem~\ref{main}, where initially regular sequences and regular sequences can be combined to get longer initially regular sequences, and where the reuse of variables in the construction of initially regular sequences is possible.

We begin by showing that under suitable assumptions initially regular sequences remain initially regular after enlarging the ideal appropriately.

\begin{proposition}\label{append a seq}
Let $I$ be a monomial ideal in a polynomial ring $R$. Let $B = \{b_{i,j} \mid 1 \leq i \leq q, 0 \leq j \leq t_i\}$ be distinct variables in $R$ satisfying the conditions of Theorem~\ref{main}. Let $f_i = \sum_{j=0}^{t_i} b_{i,j}$, for $i = 1, \ldots, q$. Let $Y = \{y_1, \ldots, y_r\}$ be a collection of variables in $R$ that is disjoint from $B$, and let $h_1, \ldots, h_{\ell} \in k[y_1, \ldots, y_r] \subseteq R$. Then $f_1, \ldots, f_q$ is an initially regular sequence on $R/(I, h_1, \ldots, h_{\ell})$.
\end{proposition}

\begin{proof}
We prove the statement by applying Theorem~\ref{main} to the ideal $H=\ini(I,h_1,\ldots, h_{\ell})$. Let  $K=(I,h_1, \ldots, h_{\ell})$ and notice that $K$ is an $(R_1, R_2)$-factorable  ideal, where $R_1=k[x_1, \ldots , x_u]$, $R_2=k[y_1, \ldots , y_r]$, and $B \subseteq \left\{x_1, \ldots, x_u\right\}$. Then $d_{b_{i,j}}(H) \leq d_{b_{i,j}}(I) \leq 1$ for all $i \geq 1$ and $j \geq 1$ since $B \cap Y =\emptyset$, by Corollary~\ref{degrees}. Hence, condition (1) of Theorem~\ref{main} is satisfied.

To see that condition (2) is satisfied, let $N$ be a monomial generator of $H$ such that $b_{i,0} \mid N$ for some $i$. Then $N=\lcm(M_{i_1}, \ldots, M_{i_e})g$, where $\{M_{i_1}, \ldots, M_{i_e}\}$ is a subcollection of $R_1$-factors of the minimal generators of $K$ (as in Proposition~\ref{GrobnerArgument}) and $g \in R_2$. Since $b_{i,0} \not\in \{y_1, \ldots, y_r\}$, we have that $b_{i,0} \mid \lcm(M_{i_1}, \ldots, M_{i_e})$. Moreover, since $h_i \in R_2$, then  $b_{i,0} \mid M_{i_v}$ for some $M_{i_v}$ that is a factor of a generator of $I$. By Proposition~\ref{GrobnerArgument}, we may assume that  $N_{i_v}=M_{i_{v}}g_{i_{v}}$ is the corresponding monomial generator of $I$ with $g_{i_{v}} \in R_2$. Hence, by our assumptions on  $f_1, \ldots, f_q$, there must exist $j > 0$ such that $b_{i,j} \mid N_{i_v}$. Therefore, $b_{i,j} \mid M_{i_{v}}$ since $B \cap Y =\emptyset$. Hence, condition (2) of Theorem~\ref{main} is satisfied, and the conclusion now follows.
\end{proof}

If in addition to the assumptions of Proposition~\ref{append a seq} we assume that the sequence $h_1, \ldots, h_{\ell}$ is a regular sequence on $R/I$ then we can get a better bound on the depth.

\begin{corollary}\label{append reg seq}
Let $I$, $B$, and $Y$ be as in Proposition~\ref{append a seq}. Suppose further that $h_1, \ldots, h_{\ell}$ is a regular sequence on $R/I$. Then $\depth R/I \geq \ell+q$.
\end{corollary}

\begin{proof}
By Proposition~\ref{ini reg bound} and Proposition~\ref{append a seq}, $\depth R/(I, h_1, \ldots, h_{\ell}) \geq q$. Moreover, since $h_1, \ldots, h_{\ell}$ is a regular sequence on $R/I$, then $\depth R/I=\depth R/(I, h_1, \ldots, h_{\ell}) +\ell$.
\end{proof}

The next corollary gives another way to obtain a longer initially regular sequence, provided that the first part consists of an initially regular sequence of elements that are sums of two variables. Notice that we do not require that all these variables need to be distinct.

\begin{corollary}\label{extend ini regular with sums of two}
Let $I$, $B$, and $Y$ be as in Proposition~\ref{append a seq}. Suppose further that $h_1, \ldots, h_{\ell}$ is an initially regular sequence with each $h_i$ of the form $y_{i_1}+y_{i_2}$, for some $y_{i_1}, y_{i_2}$ distinct variables. Then $h_1, \ldots, h_{\ell}, f_1, \ldots, f_q$ is an initially regular sequence on $R/I$.
\end{corollary}
\begin{proof}
The result follows from Proposition~\ref{append a seq} and Lemma~\ref{lem.initial}.
\end{proof}

Our next goal is to construct sequences that are both regular and initially regular. Our construction is inspired by the notion of leaves in graphs. We say that a variable $x$ is a \emph{leaf} in a monomial ideal $I$ if there exists a unique monomial generator $M \in I$ such that $x \mid M$.

\begin{remark} By employing a change of variables if needed (see \cite[Lemmas 3.3 and 3.5]{NPV}), the depth of a monomial ideal is unchanged if we assume that $d_x(I)=1$ for any leaf $x$ of $I$.
\end{remark}

\begin{lemma}\label{leaves regular}
Let $I$ be a monomial ideal in a polynomial ring $R$.  Suppose that $x$ and $y$ are two leaves in $I$ with $M_1, M_2$ the unique monomial generators in $I$ such that $x\mid M_1$ and $y\mid M_2$. Suppose there exist monomials $z, w \in R$ such that $x \nmid z$, $z \mid M_1$, $y \nmid w, w\mid M_2$, $\gcd(z,w) = 1$, and $zw \in I$.
Then $x+y$ is a regular element on $R/I$.
\end{lemma}

\begin{proof}
First notice that if $M_1=M_2$, then $x+y$ is a regular element on $R/I$, by Lemma~\ref{regular} and Remark~\ref{degree 1 relaxed}. Hence, we may assume that $M_1 \neq M_2$.

Suppose that $g(x+y) \in I$, for some $g \in R$. Then we may assume that $g$ is a monomial by Lemma~\ref{r=1 case} and Remark~\ref{degree 1 relaxed} .
Since $I$ is a monomial ideal, then $gx \in I$ and $gy \in I$. Thus, if $g \not\in I$, then $M_1\mid gx$, since $x$ appears only in $M_1$, and similarly $M_2\mid gy$. Thus, $gx=xzM_1'$ for some monomial $M_1'$ and therefore $g=zM_1'$. Similarly, $g=wM_2'$, for some monomial $M_2'$. Hence, $zw \mid g$ and therefore $g \in I$, since $\gcd(z,w)=1$ and  $zw \in I$.
\end{proof}

Notice that if $I$ is the edge ideal of a graph, unless $xy$ is an isolated edge, then the conditions $\gcd(z,w)=1$ and that $zw\in I$ in Lemma~\ref{leaves regular} means that the two leaves we are considering are distance three apart. Moreover, the result does not hold in general if the distance is not three as can be seen in the next example.

\begin{example}
Let $I=(x_1x_2,x_2x_3,x_2x_4,x_4x_5,x_5x_6,x_5x_7)\subseteq R=\mathbb{Q}[x_1,\ldots ,x_7]$ be the edge ideal of the graph $G$ depicted below.

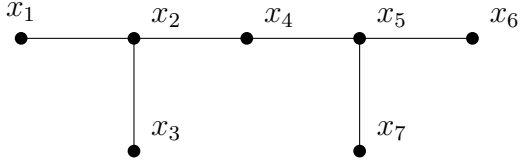
\setlength{\unitlength}{0.7cm}
\begin{tikzpicture}

    \tikzstyle{point}=[circle,thick,draw=black,fill=black,inner sep=0pt,minimum width=4pt,minimum height=4pt]

 \node (a)[point,label={[label distance=0cm]:$x_1$}] at (0.5,2) {};

    \node (b)[point,label={[label distance=0cm]3:$x_2$}] at (2,2) {};

\node (d)[point,label={[label distance=0cm]3:$x_4$}] at (3.5,2) {};

 \node (e)[point,label={[label distance=0cm]3:$x_5$}] at (5,2) {};

 \node (g)[point,label={[label distance=0cm]3:$x_7$}] at (5,0.5) {};

    \node(c)[point,label={[label distance=0cm]3:$x_3$}] at (2,0.5) {};

   \node(f)[point,label={[label distance=0cm]3:$x_6$}] at (6.5,2) {};

\draw (a.center) -- (b.center);

\draw(b.center)--(d.center);
\draw(d.center)--(e.center);
\draw(b.center)--(c.center);
\draw (e.center) -- (f.center);
\draw (e.center) -- (g.center);
\end{tikzpicture}

Notice that $x_1, x_3,x_6$, and $x_7$ are all leaves in $I$, no two of which are distance three apart. It can be checked that no sum of any two of these leaves is a regular element.
\end{example}

\begin{remark}\label{leaf pairs further assumptions}
Under the assumptions of Lemma~\ref{leaves regular} we may assume as in the proof that $M_1 \neq M_2$. Moreover, we may assume that $zw$ is a minimal generator of $I$. Indeed,
since $zw \in I$, then $zw=MN$, where $M$ is a monomial generator of $I$ and $N \in R$ is another monomial. Let $z=z'z''$ and $w=w'w''$, with $z'\mid M$, $w'\mid M$, and $\gcd(z'', M)=\gcd(w'',M)=1$. Since $\gcd(z,w)=1$, then $\gcd(z',w')=1$, and therefore $M=z'w'$.

Finally, since $x \nmid z$, then $x \nmid z'$ and similarly, $y\nmid w'$. Also, since $z \mid M_1$, then $z' \mid M_1$ and similarly, $w' \mid M_2$. Therefore, we may replace $z$ and $w$ by $z'$ and $w'$, respectively and assume that $zw$ is indeed a minimal monomial generator of $I$.
\end{remark}

\begin{definition}\label{leaf pair}
An ordered pair of leaves $x,y$ of a monomial ideal $I$ which satisfies the conditions of Lemma~\ref{leaves regular} with $M_1 \neq M_2$ is called a \emph{leaf pair}. We will say that two leaf pairs $x,y$ and $a, b$ are \emph{disjoint} if $\{x,y\} \cap \{a,b\}=\emptyset$.
\end{definition}

We are now ready to show that disjoint leaf pairs can be used to form an initially regular sequence. Using Theorem~\ref{regular equiv ini regular}, we see that the sequence is also a regular sequence.

\begin{theorem}\label{leaf pairs regular}
Let $I$ be a monomial ideal in a polynomial ring $R$. Let $\{x_i,y_i\}_{i=1}^{\ell}$ be a set of disjoint leaf pairs with respect to $I$. Then $x_1+y_1, \ldots, x_{\ell}+y_{\ell}$ is both a regular sequence and an initially regular sequence on $R/I$ with respect to any term order such that $x_i > y_i$ for all $i$.
\end{theorem}

\begin{proof} We start with the case where $\ell = 2$. For ease of notation, let $x,y$ and $a,b$ denote the given two leaf pairs (with $x > y$ and $a > b$). Let $M_1, M_2, N_1$, and $N_2$ be the monomial generators of $I$ that are divisible by $x,y,a$, and $b$, respectively.

By the definition of a leaf pair and Lemma~\ref{leaves regular}, $x+y$ is regular on $R/I$. Since $I$ is monomial, $x+y$ is also initially regular.
It remains to show that $a+b$ is regular on $R/(I,x+y)$ and on $R/\ini(I,x+y)$. By Theorem~\ref{regular equiv ini regular} it is enough to show that $a+b$ is regular on $R/\ini(I, x+y)$. As in Theorem~\ref{regular equiv ini regular}, $R/\ini(I,x+y) \cong R'/I'$, where $R=k[x,y,a,b,x_5, \ldots , x_n]$, $R'=k[y,a,b,x_5, \ldots, x_n]$ and $I'=\langle \ini(I, x+y)\setminus x \rangle= \langle I\setminus \{M_1\} \cup \{\widehat{M_1}\} \rangle$,  where $\widehat{M_1}={\frac{-y^{d_x(M_1)}}{x^{d_x(M_1)}}}M_1$. The isomorphism is induced by the map $\phi: R \rightarrow R'$ that sends $x$ to $-y$ and fixes all other variables. In light of Lemma~\ref{leaves regular}, it suffices to show that $\{\phi(a),\phi(b)\}$ is a leaf pair with respect to the ideal $I'$ in $R'$.

By the definition of a leaf pair, $a \mid N_1$, $b\mid N_2$, and $N_1 \not=N_2$. In addition, since $b$ is a leaf, then $b \nmid N_1$. Since $\phi(b)=b$, we have $\phi(b) \mid \phi(N_2)$ but $\phi(b) \nmid \phi(N_1)$, so $\phi(N_1) \not= \phi(N_2)$. Also, there exist monomials $\alpha \vert N_1$ and $\beta \vert N_2$ with $a\nmid \alpha$ and $b \nmid \beta$ with $\alpha \beta \in I$ and $\gcd(\alpha, \beta)=1$.
Notice that since $\alpha \mid N_1$ and $\beta \mid N_2$, we have $\phi(\alpha) \mid \phi(N_1)$ and $\phi(\beta) \mid \phi(N_2)$.  Since $\alpha \beta \in I$ and, as in the proof of Theorem~\ref{regular equiv ini regular}, $\phi(I) \subseteq I'$, we have $\phi(\alpha)\phi(\beta) \in I'$.  Since $a\neq x$, we have $\phi(a)=a$. Observe that if $\phi(a) \mid \phi(\alpha)$ then we must have $\phi(\alpha) \not=\alpha$, so $a ~\big|~ \frac{y^{d_x(\alpha)}}{x^{d_x(\alpha)}}\alpha$. This implies that $a \mid \alpha$,  since $a \neq y$,  which is a contradiction. Thus, $\phi(a) \nmid \phi(\alpha)$. Similarly, $\phi(b) \nmid \phi(\beta)$.

It remains to show that $\gcd(\phi(\alpha), \phi(\beta))=1$.
If $x \nmid \alpha$ and $x \nmid \beta$, then $\phi(\alpha)=\alpha$ and $\phi(\beta)=\beta$ and therefore $\gcd(\phi(\alpha), \phi(\beta))=1$.  Otherwise, since $\gcd(\alpha, \beta)=1$, $x$ can divide at most one of $\alpha$ and $\beta$. Without loss of generality, suppose that $x \mid \alpha$. Then, $\phi(\alpha)=\frac{-y^{d_x(\alpha)}}{x^{d_x(\alpha)}}\alpha$ and $\phi(\beta)=\beta$.  By Remark~\ref{leaf pairs further assumptions}, we may assume that $\alpha \beta $ is a minimal generator of $I$, and therefore at most one of $x$ or $y$ divides $\alpha \beta$. Since we have assumed $x \mid \alpha$, it follows that $x, y \nmid \beta$. Therefore, $\gcd(\phi(\alpha), \phi(\beta))=\gcd(\frac{-y^{d_x(\alpha)}}{x^{d_x(\alpha)}}\alpha, \beta)=\gcd(-y^{d_x(\alpha)}\alpha, \beta)=1$.

To see the general case, for any $\ell \ge 2$, note that $\{\phi(x_i), \phi(y_i)\}_{i=2}^\ell$ is a set of disjoint leaf pairs with respect to $I' \subseteq R'$, and so the conclusion follows by induction.
\end{proof}

Our next result shows that the initially regular sequences formed by Theorem~\ref{main} can be combined with leaf pairs to create longer initially regular sequences, and thus improve the depth bound.

\begin{corollary} \label{pairs of leaves}
Let $I$ be a monomial ideal in a polynomial ring $R$. Suppose that $B = \{b_{i,j} \mid 1 \leq i \leq q, 0 \leq j \leq t_i\}$ are distinct variables in $R$ satisfying the conditions of Theorem~\ref{main}, and  let $f_i = \sum_{j=0}^{t_i} b_{i,j}$ for $i = 1, \ldots, q$. Suppose that there exist a sequence of disjoint pairs of leaves $\{x_k, y_k\}_{k=1}^{\ell}$ in $I$ such that each pair $x_k, y_k$ satisfies the conditions of Lemma~\ref{leaves regular}. Assume further that $B \cap \{x_1, y_1, \ldots, x_{\ell}, y_{\ell}\}=\emptyset$. Then $x_1+y_1, \ldots, x_{\ell}+y_{\ell}, f_1,f_2, \ldots, f_q$ is an initially regular sequence on $R/I$ with respect to any term order such that $x_k > y_k$ for all $k\leq \ell$ and $b_{i,0} > b_{i,j}$ for $i<q$ and $j \leq t_i$. Particularly, $\depth R/I \geq \ell+q$.
\end{corollary}

\begin{proof} By Theorem~\ref{leaf pairs regular}, $x_1+y_1, \ldots, x_\ell+y_\ell$ forms an initially regular sequence on $R/I$. By Proposition~\ref{append a seq}, $f_1, \ldots, f_q$ is an initially regular sequence on $R/(I,x_1+y_1, \ldots, x_\ell+y_\ell)$. Now, by Lemma~\ref{lem.initial}, we have
$$\ini(I,x_1+y_1, \ldots, x_\ell + y_\ell) = \ini(\ini(\ldots \ini(\ini(I,x_1+y_1),x_2+y_2), \ldots), x_\ell+y_\ell).$$
Thus, we can concatenate $x_1+y_1, \ldots, x_\ell+y_\ell$ and $f_1, \ldots, f_q$ to get an initially regular sequence on $R/I$. The last claim follows from Proposition~\ref{ini reg bound}.
\end{proof}

The following examples illustrate Corollary~\ref{pairs of leaves} in the special case of edge ideals of graphs. For graphs, our bounds are similar to known bounds, but our results give a regular sequence or an approximation of one that achieves the bound. For hypergraphs in general, our results are significantly better than known results, and still produce a regular sequence or an approximation of one.

\begin{example} \label{H tree}
Let $R=\mathbb{Q}[a,b,c,d,e,f]$, let $I=(ab,bc,be,de,ef)$ be the edge ideal of the graph depicted below, and fix a term order with $a>b>c>d>e>f$.

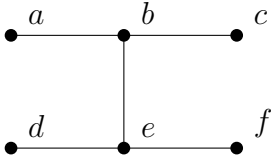
\setlength{\unitlength}{0.7cm}
\begin{tikzpicture}

    \tikzstyle{point}=[circle,thick,draw=black,fill=black,inner sep=0pt,minimum width=4pt,minimum height=4pt]

 \node (a)[point,label={[label distance=0cm]3:$a$}] at (0.5,2) {};

    \node (b)[point,label={[label distance=0cm]3:$b$}] at (2,2) {};

\node (c)[point,label={[label distance=0cm]3:$c$}] at (3.5,2) {};

    \node (d)[point,label={[label distance=0cm]3:$d$}] at (0.5,0.5) {};

    \node(e)[point,label={[label distance=0cm]3:${e}$}] at (2,0.5) {};

   \node(f)[point,label={[label distance=0cm]3:$f$}] at (3.5,0.5) {};

    \draw (b.center) -- (e.center);

\draw (a.center) -- (b.center);
\draw(d.center)--(e.center);
\draw(b.center)--(c.center);
\draw(e.center)--(f.center);

\end{tikzpicture}

Notice that $a+b, d+e$ is an initially regular sequence on $R/I$ by Theorem~\ref{main}. Also, $c, f$ is a leaf pair in the sense of Definition~\ref{leaf pair}. Hence, by Corollaries~\ref{regular sums of two} and \ref{pairs of leaves}, we have that $c+f, a+b, d+e$ is both a regular and an initially regular sequence on $R/I$. Therefore, $\depth R/I \geq 3$. By \cite[Theorem~1.1]{KT} we have that $\depth R/I=6-{\rm{bight}} I=3$. The known bound in \cite[Corollary~4.2]{DS} gives  $\depth R/I \geq 2$. In this example, Corollary~\ref{pairs of leaves} provides an optimal construction in the sense that it produces a maximal regular and initially regular sequence on $R/I$.

\end{example}

Note that in Example~\ref{H tree} we have a tree and, in this case, $\depth R/I$ is determined by \cite[Theorem~1.1]{KT}. The next example is a slight modification of the previous one.

\begin{example} \label{pair leaves example}
Let $R=\mathbb{Q}[a,b,c,d,e,f,g]$, let $I=(ab,ae,bc,be,de,ef,bg)$ be the edge ideal of the graph depicted below, and fix a term order with $a>b>c>d>e>f>g$.

\setlength{\unitlength}{0.7cm}
\begin{tikzpicture}

    \tikzstyle{point}=[circle,thick,draw=black,fill=black,inner sep=0pt,minimum width=4pt,minimum height=4pt]

 \node (a)[point,label={[label distance=0cm]3:$a$}] at (0.5,2) {};

    \node (b)[point,label={[label distance=0cm]3:$b$}] at (2,2) {};

\node (c)[point,label={[label distance=0cm]3:$c$}] at (3.5,2) {};

    \node (d)[point,label={[label distance=0cm]3:$d$}] at (0.5,0.5) {};

    \node(e)[point,label={[label distance=0cm]3:${e}$}] at (2,0.5) {};

   \node(f)[point,label={[label distance=0cm]3:$f$}] at (3.5,0.5) {};

\node (g)[point,label={[label distance=0cm]1:$g$}] at (2,3.5) {};

    \draw (b.center) -- (e.center);

\draw (a.center) -- (b.center);
\draw(d.center)--(e.center);
\draw(b.center)--(c.center);
\draw(e.center)--(f.center);
\draw(b.center)--(g.center);
\draw(a.center)--(e.center);
\end{tikzpicture}

By \cite[Corollary~4.2]{DS}, $\depth R/I \geq 2$. Also the diameter of this graph is $3$ and hence $\depth R/I \geq 2$  by \cite[Theorem~3.1]{FM}. Using Macaulay~2~\cite{M2} we have that $\depth R/I=3$. Notice that $d+g, c+f, a+b+e$ is both a regular sequence and an initially regular sequence on $R/I$ by Corollary~\ref{regular sums of two}, Theorem~\ref{main}, and Corollary~\ref{pairs of leaves}. Our results, again, provide a sharp bound for depth as well as a sequence that realizes the depth.
\end{example}

Next we exhibit a special situation where a variable can be reused in the creation of initially regular sequences.

\begin{theorem} \label{iterate repeated neighbors}
Let $I$ be a monomial ideal in a polynomial ring $R$. Suppose that $b_0, \ldots, b_{t}$ are distinct variables in $R$  and $>$ is a fixed term order such that  $b_0 > b_1 > \cdots > b_{t}$. Suppose that for some $q \le t$, the sets $\{b_{0},b_{1}\}, \{b_1, b_2\}, \ldots, \{b_{q-2}, b_{q-1}\}, \{b_{q-1}, b_q, \ldots, b_{t}\}$ satisfy the conditions $(1)$ and $(2)$ of Theorem~\ref{main}. Let $f_i = b_{i-1}+b_i$, for $1 \le i \le q-1$, and $f_q=b_{q-1}+\ldots+b_{t}$. Then $f_1, \ldots, f_q$ is both a regular and an initially regular sequence on $R/I$.

\end{theorem}

\begin{proof}
By Corollary~\ref{regular sums of two} it suffices to show that  $f_1, f_2, \ldots, f_q$ is an initially regular sequence on $R/I$. We will proceed by induction. When $q=1$ the result follows from Theorem~\ref{main}.

By induction, it suffices to show that $\{b_1,b_2\}, \ldots , \{b_{q-1} ,\ldots , b_t\}$ satisfy the conditions $(1)$ and $(2)$ of Theorem~\ref{main} applied to  $H=\ini(I, f_1)=\ini(I, b_0+b_1)$. By Corollary~\ref{degrees}, $d_{b_i}(H) \leq 1$ for $i \geq 2$, so condition $(1)$ of Theorem~\ref{main} holds for the sets $\{b_1,b_2\}, \ldots, \{b_{q-1}, \ldots , b_t\}$ relative to $H$.

Let $I=(M_1, \ldots , M_p)$, where $M_1, \ldots, M_p$ is a minimal set of monomial generators of $I$ such that  $b_0 \mid M_i$ if and only if $1\leq i \leq \ell$. By Lemma~\ref{gens H}, $H=\left(b_0, \widehat{M_1}, \ldots , \widehat{M_{\ell}}, M_{\ell +1}, \ldots , M_p \right)$, where $\widehat{M_i}=b_1^{d_{b_0}(M_i)} \frac{M}{b_0^{d_{b_0}(M)}}$.

By the definition of $\widehat{M}$, for $j \not= 0,1$, $b_j\mid M$ if an only if $b_j \mid \widehat{M}$. Thus, condition $(2)$ of Theorem~\ref{main} on $\{b_2, b_3\}, \ldots, \{b_{q-1}, \ldots , b_t\}$ follows from the original hypotheses. If $b_1\mid M_i$, for any $i$, by hypothesis, $b_2 \mid M_i$. If $b_1 \mid \widehat{M_i}$ for some $i$, then by definition $b_0\mid M_i$ and again by hypothesis $b_1 \mid M_i$.  It follows that $b_2$ divides both $M_i$ and $\widehat{M_i}$ and so condition $(2)$ of Theorem~\ref{main} holds for $\{b_1,b_2\}$. The result now follows.
\end{proof}

It is interesting to note that in the situation of Theorem~\ref{iterate repeated neighbors} the given set of generators for $H$ is a minimal generating set. In the next Theorem we show that various combinations of initially regular sequences could be combined to give longer initially regular sequences.

\begin{theorem}\label{generalized iterate}
Let $I$ be a monomial ideal in a polynomial ring $R$. Let $f_1, \ldots, f_{q}$ be an initially regular sequence on $R/I$ as in Theorem~\ref{iterate repeated neighbors} with $f_1, \ldots, f_q \subset k[b_0, \ldots, b_t]$. Let $g_1, \ldots, g_{\ell}$ be an initially regular sequence on $R/I$ satisfying the conditions of Theorem~\ref{main} and assume that $g_1, \ldots, g_{\ell} \in k[y_1, \ldots, y_{r}]$, where $\{y_1, \ldots, y_{r}\} \cap \{b_0, \ldots, b_t\}=\emptyset$. Then any sequence obtained by merging a subsequence of $f_1, \ldots, f_q$  and a subsequence of $g_1, \ldots, g_{\ell}$ in any order such that whenever $f_i$ and $f_j$ appear, then $f_i$ precedes $f_{j}$ for all $j \geq i$, is an initially regular sequence on $R/I$. 
\end{theorem}

\begin{proof}
Let $\alpha_1, \ldots, \alpha_u$ be a sequence, where each $\alpha_i$ is either $f_{j}$ or $g_{{j}}$ for some $j$. Let $I_1=I$ and $I_i=\ini(I_{i-1}, \alpha_{i-1})$ for all $ 2\leq i \leq u$. It suffices to show that $\alpha_{u}$ is regular on $R/I_{u}$. We will show that the conditions of Lemma~\ref{regular} are satisfied. To simplify notation let us write $\alpha_i=z_{i,0}+\ldots +z_{i, t_i}$, where either $z_{i,j}\in\{b_0,\ldots,b_t\}$ for all $j$ or $z_{i,j}\in \{y_1,\ldots,y_r\}$ for all $j$ as appropriate. By applying Lemma~\ref{degreesFixed} repeatedly we have $d_{z_{u,j}}(I_{u}) \leq d_{z_{u,j}}(I)\leq 1$, for all $ 1\leq j \leq t_u$, by our assumptions.  Hence, condition $($a$)$ of Lemma~\ref{regular} is satisfied.

Suppose that $a_u=g_i$ for some $i$. Let $M$ be a minimal monomial generator of $I_u$ such that  that $z_{u,0} \mid M$. By repeated use of Proposition~\ref{GrobnerArgument} with $R_1=k[z_{u,0}, \ldots, z_{u,t_u}]$, we can assume that $z_{u,0} \mid N$, where $N$ is an $R_1$-monomial factor of a minimal generator of $I_1=I$ and $N \mid M$. Thus by our assumption, $z_{u,0}\mid N$ implies $z_{u,j} \mid N$ for some $j$. Therefore, $z_{u,j}\mid M$ and condition $(b)$ of Lemma~\ref{regular} is satisfied.

Now suppose $\alpha_u=f_i$ for some $i$ and $M$ is a minimal monomial generator of $I_u$ such that $z_{u,0}=b_{i-1} \mid M$. If $\alpha_{j}\neq f_{i-1}$ for all $j<u$, then the result follows as in the case $\alpha_u=g_i$. Suppose then that $\alpha_{j}=f_{i-1}$ for some $1\leq j\leq u-1$. By applying Proposition~\ref{GrobnerArgument} repeatedly we have that $b_{i-1} \mid N$, where $N$ is an $R_1$-monomial factor of a minimal monomial generator of $I_{j+1}$. By Lemma~\ref{gens H} either $b_{i-1}$ or $b_{i-2}$ divides $L$, where $L$ is a minimal monomial generator of $I_j$. Continuing in this manner using Proposition~\ref{GrobnerArgument} and Lemma~\ref{gens H} as appropriate we have that, for some $w\leq i-1$,  $b_{w}$ divides $N'$, where $N'$ is an $R_1$-monomial factor of a minimal generator of $I$. 
By our assumptions if $b_{w} \mid N'$, then $z_{u,0}=b_{i-1} \mid N'$ as well. Hence by our assumptions on $f_i$, there exists $v>0$ such that $z_{u,v} \mid N'$. Therefore, $z_{u,v}\mid M$ and condition~$($b$)$ of Lemma~\ref{regular} is again satisfied.
\end{proof}

Notice that Theorem~\ref{generalized iterate} can be extended to allow multiple initially regular sequences to be merged. The proof follows in the same manner.

\begin{remark}\label{longer sequence remark}
Let $\underline{f_1}, \ldots, \underline{f_s}$ be a collection of initially regular sequences that satisfy the conditions of Theorem~\ref{iterate repeated neighbors}.  Let $\underline{g}$ be an initially regular sequence that satisfies the conditions of Theorem~\ref{main}.  Suppose that the variables in each of the sequences $\ul{f_1}, \ldots, \underline{f_s}, \ul{g}$ are disjoint from each other. Then any sequence obtained by merging any subsequences of $\ul{f_1}, \ldots, \ul{f_q}$  and a subsequence of $\ul{g}$ in any order such that whenever $f_{i,j}$ and $f_{i,r}$ appear, then $f_{i,j}$ precedes $f_{i,r}$ for all $r \geq j$, is an initially regular sequence on $R/I$.
\end{remark}

Before proving the final result of the section, we give a series of examples. In the first example, as an immediate application of Theorem~\ref{iterate repeated neighbors}, we obtain a sharp bound for the depth of a tetrahedron. It is worth noting that none of the previously known combinatorial bounds were able to capture the exact value for this example.

\begin{example}\label{tetrahedron}
Let $R=\mathbb{Q}[a,b,c,d]$ and let $I=(abcd)$ be the edge ideal corresponding to the hypergraph of a tetrahedron depicted below.

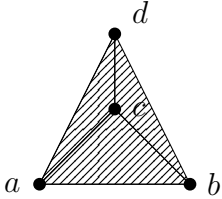
\begin{tikzpicture}

    \tikzstyle{point}=[circle,thick,draw=black,fill=black,inner sep=0pt,minimum width=4pt,minimum height=4pt]

 \node (a)[point,label={[label distance=-0.7cm]1:$a$}] at (1,0) {};

    \node (b)[point,label={[label distance=0cm]1:$b$}] at (3,0) {};

\node (c)[point,label={[label distance=0cm]1:${c}$}] at (2,1) {};

    \node (d)[point,label={[label distance=0cm]3:$d$}] at (2,2) {};

    \draw[pattern=north east lines] (a.center) -- (c.center) -- (b.center) -- cycle;

    \draw[pattern=north east lines] (a.center) -- (c.center) -- (d.center) -- cycle;

    \draw[pattern= north east lines]   (b.center) -- (c.center) -- (d.center) -- cycle;

\draw(a.center)--(c.center);
\draw(a.center)--(d.center);
\draw(a.center)--(b.center);

\end{tikzpicture}

It is easy to see that $\depth R/I=3$. However, the  known combinatorial bound of \cite[Theorem~3.2]{DS2} gives at most $\depth R/I \geq 1$.  It follows immediately by Theorem~\ref{iterate repeated neighbors} that $a+b, b+c, c+d$ is both a regular and an initially regular sequence on $R/I$ with respect to a term order with $a>b>c>d$, and that $\depth R/I\geq 3$.
\end{example}

In the next example we consider the case of the edge ideal of an octagon. It is worth noting here that we can exhibit a regular sequence that accurately computes the depth, however the fact that the last term of the sequence is regular on the appropriate module does not follow from any of our results. Therefore, there are other regular and initially regular sequences that one can compute and more work can be done in the direction of fully understanding how to construct such sequences.

\begin{example}\label{octagon}
 Let $I=(x_1x_2,x_2x_3,x_3x_4,x_4x_5,x_5x_6, x_6x_7,x_7x_8,x_1x_8) \subseteq R = \mathbb{Q}[x_1, \ldots, x_8]$
be the edge ideal of the graph  of the octagon depicted below.

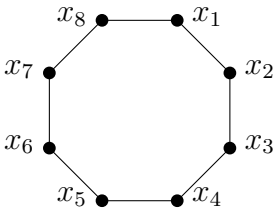
\begin{tikzpicture}

    \tikzstyle{point}=[circle,thick,draw=black,fill=black,inner sep=0pt,minimum width=4pt,minimum height=4pt]

 \node (a)[point,label={[xshift=0.4cm, yshift=-0.3 cm]: $x_1$}] at (1.7,2.7) {};

    \node (b)[point,label={[xshift=0.4cm, yshift=-0.3 cm]: $x_2$}] at (2.4,2) {};

\node (c)[point,label={[xshift=0.4cm, yshift=-0.3 cm]: $x_3$}] at (2.4,1) {};

    \node (d)[point,label={[xshift=0.4cm, yshift=-0.3 cm]: $x_4$}] at (1.7,.3) {};

\node (e)[point,label={[xshift=-0.4cm, yshift=-0.3 cm]: $x_5$}] at (0.7,0.3) {};

    \node (f)[point,label={[xshift=-0.4cm, yshift=-0.3 cm]: $x_6$}] at (0,1) {};
\node (g)[point,label={[xshift=-0.4cm, yshift=-0.3 cm]: $x_7$}] at (0,2) {};

    \node (h)[point,label={[xshift=-0.4cm, yshift=-0.3 cm]: $x_8$}] at (0.7,2.7) {};

\draw(a.center) -- (b.center) -- (c.center) -- (d.center) -- (e.center) -- (f.center) -- (g.center) -- (h.center) -- (a.center);
\end{tikzpicture}

First we note that $\depth R/I=3$. Using Theorem~\ref{main} we can only create a maximal initially regular sequence of length two on $R/I$. For example, let $f=x_2+x_1+x_3$ and $g= x_5+x_6+x_4$ and notice that $f, g$  is an initially regular sequence on $R/I$ with respect to any term order such that $x_2>x_1>x_3>x_5>x_6>x_4$. Moreover, $f, g$ is a regular sequence on $R/I$ as can be verified by Macaulay~2~\cite{M2}.

In search for a third element to complete our regular sequence we note that the only variables that were not used are $x_7,x_8$. But neither $x_7+x_8+x_6$ nor $x_8+x_7+x_1$ are regular on $R/(I, f,g)$ or initially regular on $R/(\ini(I,f),g)$.
However, using Macaulay~2~\cite{M2} for instance we can see that $h=x_7+x_8+x_6+x_1$ is regular on $R/(I, f,g )$. Moreover, $f,g,h$ is both a regular and an initially regular sequence on $R/I$ with respect to the any term order such that $x_7>x_2>x_1 >x_3 > x_5>x_6>x_4>x_8$.
\end{example}

In the next example, we shall see that when there is a freedom of choice in Algorithm~\ref{alg}, our bound on the depth can at times be made to be the actual value.

\begin{example} \label{hypergraph example}
Let $R=\mathbb{Q}[a,b,c,d,e,f,g,h]$ and let $I=(abc,acd,bcd,de,efgh)$ be the edge ideal of the following hypergraph.

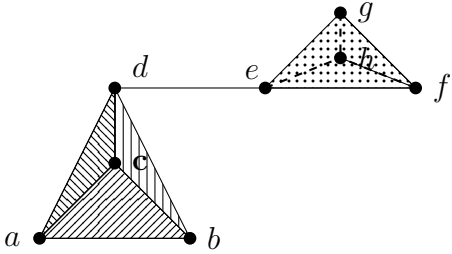
\begin{tikzpicture}

    \tikzstyle{point}=[circle,thick,draw=black,fill=black,inner sep=0pt,minimum width=4pt,minimum height=4pt]

 \node (a)[point,label={[label distance=-0.7cm]1:$a$}] at (1,0) {};

    \node (b)[point,label={[label distance=0cm]1:$b$}] at (3,0) {};

\node (c)[point,label={[label distance=0cm]1:$\bf{c}$}] at (2,1) {};

    \node (d)[point,label={[label distance=0cm]3:$d$}] at (2,2) {};

    \node(e)[point,label={[label distance=-0.5cm]3:$e$}] at (4,2) {};

   \node(f)[point,label={[label distance=0cm]1:$f$}] at (6,2) {};

   \node(g)[point,label={[label distance=0cm]1:$g$}] at (5,3) {};

   \node(h)[point,label={[label distance=0cm]1:$h$}] at (5,2.4) {};

    \draw[pattern=north east lines] (a.center) -- (c.center) -- (b.center) -- cycle;

    \draw[pattern=north west lines] (a.center) -- (c.center) -- (d.center) -- cycle;

    \draw[pattern=vertical lines]   (b.center) -- (c.center) -- (d.center) -- cycle;

    \draw[pattern=dots] (e.center) -- (f.center) -- (g.center) --(e.center) -- cycle;
\draw[thick, dashed] (f.center) -- (h.center);
\draw[thick, dashed] (e.center) --(h.center);
\draw[thick, dashed](g.center)--(h.center);

    \draw (d.center) -- (e.center);

\draw (f.center) -- (g.center);
\draw (f.center)--(e.center);
\draw (f.center)--(h.center);
\draw(a.center)--(c.center);
\draw(a.center)--(d.center);
\draw(a.center)--(b.center);

\end{tikzpicture}

Note that $f+h, h+g, g+e, a+c, c+b+d$ is both a regular and an initially regular sequence on $R/I$ by Theorems~\ref{iterate repeated neighbors} and \ref{longer sequence remark} with respect to any order such that $f>h>g>e$, $a>c>b$, and $c>d$. Hence, $\depth R/I \geq 5$ and computations on Macaulay~2~\cite{M2} show that this is actually an equality.
\end{example}

The final result of this section shows that the method of creating initially regular sequences produces a bound that can be effectively combined with the use of polarization when bounding the depths of non-squarefree monomial ideals. That is, the bound produced will be sufficiently large to at least recover the number of polarizing variables. Note that the prior known depth bound for general hypergraphs, using dominating parameters, is not generally effective when combined with this technique due to the nature of polarization. By definition, hyperedges of the polarization that contain polarizing variables will also contain the corresponding original variables, creating a situation where it is relatively easy for a few edges to dominate many others.

\begin{theorem}\label{polarization}
Let $I$ be a monomial ideal in $R=k[x_1, \ldots, x_n]$ and let $I^{\pol}\subset R^{\pol}$ be its polarization. Then the maximal length of an initially regular sequence on $R^{\pol}/I^{\pol}$ is at least $\sum_{i=1}^n(d_{x_i}(I)-1)$, which is the number of polarizing variables.
\end{theorem}

\begin{proof}
Set $d_i=d_{x_i}(I)$. Then $x_i$ is polarized by variables $x_i, x_{i,2}, x_{i,3}, \ldots , x_{i,d_i}$. Set $x_{i,1}=x_i$ for ease of notation. Let $R^{\pol}=R[x_{i,1}, \ldots, x_{i,d_i} \mid 1\leq i \leq n]$ and let $I^{\pol}$ denote the polarization of $I$ in $R^{\pol}$. Then by the definition of polarization, for any $1\leq i\leq n$, $d_{x_{i,j}}(I^{\pol})=1$ for all $1\leq j \leq d_i$ and  if $x_{i,j}$ divides a monomial generator $M$ of $I^{\pol}$, then $x_{i,k}$ divides $M$ for all $1\leq k \leq j$. In particular, the sets $\{ x_{i,j}, x_{i,j-1}\}$ satisfy the conditions of Theorem~\ref{iterate repeated neighbors} for $2 \leq j \leq d_i$. By Theorem~\ref{iterate repeated neighbors}, the elements
$$x_{i,d_i}+x_{i,d_i -1}, x_{i,d_i -1}+x_{i,d_i -2}, \ldots, x_{i,2}+x_{i,1}$$
form an initially regular sequence on $R^{\pol}/I^{\pol}$ with respect to an appropriate term order. 
By Theorem~\ref{generalized iterate} and Remark~\ref{longer sequence remark} we have that 
$$x_{1,d_1}+x_{1,d_1 -1}, x_{1,d_1 -1}+x_{1,d_1 -2}, \ldots, x_{1,2}+x_{1,1}, x_{2,d_2}+x_{2,d_2-1}, \ldots , x_{n,2}+x_{n,1}$$
form an initially regular sequence on $R^{\pol}/I^{\pol}$ with respect to an appropriate term order. 
\end{proof}

Theorem~\ref{polarization} illustrates the power of the choices made when forming initially regular sequences. The goal is to produce the longest possible initially regular sequence by a judicious choice of elements satisfying the hypotheses of Theorem~\ref{main} and its extensions. When this maximal length is greater than the minimum guaranteed by Theorem~\ref{polarization}, a positive lower bound for the depth of the original monomial ideal results.

\begin{example}\label{path}
Let $R=\mathbb{Q}[a,b,c]$ and let $I=(ab,bc)$ be the edge ideal of the graph of a path of length $2$ depicted below.

\begin{tikzpicture}

    \tikzstyle{point}=[circle,thick,draw=black,fill=black,inner sep=0pt,minimum width=4pt,minimum height=4pt]

 \node (a)[point,label={[label distance=0cm]3:$a$}] at (0.5,2) {};

    \node (b)[point,label={[label distance=0cm]3:$b$}] at (2,2) {};

\node (c)[point,label={[label distance=0cm]3:$c$}] at (3.5,2) {};

 \draw (a.center) -- (b.center);
\draw(b.center)--(c.center);

\end{tikzpicture}

Consider the ideal $H=I^2=(a^2b^2, ab^2c,b^2c^2)$. Notice that since $d_{a}(I)=d_{b}(I)=d_c(I)=2$ we may not use any of our previous results to obtain any regular or initially regular elements on $R/I$.

We will use the method of polarization to obtain a bound on the depth of $R/I^2$. Let $a_1, b_1,c_1$ be polarizing variables for $a,b$, and $c$, respectively. Then
$$H^{\pol}=(aa_1bb_1, abb_1c, bb_1cc_1) \subseteq R[a_1, b_1, c_1]=R^{\pol}.$$
By Theorem~\ref{main}, Theorem~\ref{iterate repeated neighbors}, and Theorem~\ref{generalized iterate} we have that $c_1+c, a_1+a, a+b, b+b_1$ is both a regular and an initially regular sequence on $R^{\pol}/H^{\pol}$ with respect to a term order such that $c_1>c, a_1>a>b>b_1$. Hence, $\depth R^{\pol}/H^{\pol}\geq 4$ and therefore, $\depth R/I^2 \geq 4-3=1$, by \cite[Corollary~1.6.3]{HH}. Finally, we can verify that $\depth R/I^2 =1$, using Macaulay~2~\cite{M2}. Notice that the prior known depth bound for general hypergraphs yields $\depth R^{\pol}/H^{\pol}\geq 1$, which is not large enough to account for the three polarizing variables.
\end{example}

This last example shows how our results on initially regular sequences and the technique of polarization can lead to estimates on the depth of higher powers of monomial ideals. However, the bounds obtained are highly dependent on the structure of the original monomial ideal.  In our final section we return to the case of non-monomial ideals.

\section{Applications to non-monomial classes of ideals}\label{applications}

When forming an initially regular sequence on $R/I$ for a general ideal $I$, the first step of the algorithm is to find the initial ideal of $I$ with respect to a convenient term order. In Set-up~\ref{setup1}, it was assumed that this step had already been performed, thus allowing us to focus on monomial ideals. In this section, we return the focus to general ideals in polynomial rings. We provide selected examples of interesting classes of ideals for which there is a known Gr\"{o}bner basis and illustrate how our results can be applied.

\subsection{Coding Theory and Oriented Directed Graphs}

Recent work studying algebraic properties of edge ideals and weighted oriented graphs was motivated by coding theory. Reed-Muller codes are associated to certain projective spaces over finite fields. Connections between algebraic properties of the associated vanishing ideals and code invariants have been studied by a variety of authors. It was shown in \cite{RT} that the vanishing ideals of these projective spaces can be generated by a set of binomials that form a Gr\"{o}bner basis whose resulting initial ideal is precisely the edge ideal of a weighted oriented graph.  The Cohen-Macaulay property of these ideals was studied in \cite{GMSVV, Oaxaca}. Using our techniques we can construct initially regular sequences that bound the depths of the edge ideals of general weighted oriented graphs.

\begin{example}\label{oriented-directed-graph}
Let $R=\mathbb{Q}[x_1, x_2,x_3, y_1, y_2, y_3]$ and let $I=(x_1y_1, x_1^2x_2,x_2^2y_2,x_2^2x_3,x_3y_3,x_1x_3)$ be the edge ideal of the weighted oriented graph depicted below.

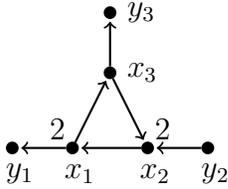
\begin{tikzpicture}[shorten >=1pt]

\tikzstyle{point}=[circle,thick,draw=black,fill=black,inner sep=0pt,minimum width=4pt,minimum height=4pt]

 \node (a)[point,label={[xshift=0.1cm, yshift=-0.7 cm]:$y_1$}] at (0.2,0) {};

    \node (b)[point,label={[xshift=0.1cm, yshift=-0.7 cm]:$x_1$}] at (1,0) {};

\node (c)[point,label={[xshift=0.1cm, yshift=-0.7 cm]:$x_2$}] at (2,0) {};

\node (d)[point,label={[xshift=0.1cm, yshift=-0.7 cm]:$y_2$}] at (2.8,0) {};

\node (e)[point,label={[label distance=0cm]1:$x_3$}] at (1.5,1) {};

\node (f)[point,label={[label distance=0cm]1:$y_3$}] at (1.5,1.8) {};

\node at (.8,0.25) {$2$};
\node at (2.2,0.25) {$2$};

\path[->, draw, thick]
(b) edge (a)
(c) edge (b)
(b) edge (e)
(e) edge (f)
(e) edge (c)
(d) edge (c);

\end{tikzpicture}

 Using $a$ as a polarizing variable for $x_1$ and $b$ as a polarizing variable for $x_2$, we have $I^{\pol}=(x_1y_1, x_1ax_2,x_2by_2,x_2bx_3,x_3y_3,x_1x_3)$. By Corollary \ref{regular sums of two} and Theorem \ref{main}, $y_1+x_1, y_3+x_3$ is a regular sequence with respect to $R^{\pol}/I^{\pol}$. Combining this with Theorems \ref{main}, \ref{iterate repeated neighbors}, and \ref{generalized iterate} it follows that $y_1+x_1, y_3+x_3, y_2+x_2, x_2+a+b$ is an initially regular sequence of $R^{\pol}/I^{\pol}$. Thus, $\depth R/I \geq 4-2=2$. Note that $\dim R/I = 3$ and $R/I$ is not Cohen-Macaulay by \cite[Theorem 1.1]{Oaxaca}, so $\depth R/I = 2$. 
\end{example}

\subsection{Binomial Edge Ideals}

Although the focus of this paper has been monomial edge ideals, there is also a toric ideal $J$, also called the binomial edge ideal, associated to a graph. Starting with a generic $2 \times n$ matrix 
$$\left[ \begin{array}{ccc}
x_1 & \cdots & x_n\\
y_1 & \cdots & y_n
\end{array} \right]$$
whose columns are indexed by the vertices of a graph $G$, the binomial edge ideal is generated by the set of $2\times 2$ minors corresponding to pairs of columns $(i,j)$ whenever $(x_i,x_j) \in E(G)$. There are known Gr\"{o}bner bases and a universal Gr\"{o}bner basis for such ideals (see \cite{HHHKR}). Moreover, the initial ideals obtained from these Gr\"obner bases can be determined by examining paths in the graph. In general, one can directly apply our results to the hypergraphs associated to these initial ideals to find bounds for the depth of a binomial edge ideal. The resulting initially regular sequences often correspond to paths and leaves in the given graph.

\begin{example}\label{binomial-edge-ideal}
Let $J=(x_1y_2-x_2y_1, x_2y_3-x_3y_2, x_3y_4-x_4y_3)$ be the binomial ideal associated to the path of length $3$ in $R=k[x_1, \ldots, x_4,y_1, \ldots, y_4]$, whose standard edge ideal is given by $I=(x_1x_2,x_2x_3,x_3x_4)$ in $k[x_1, \ldots, x_4]$. By \cite[Theorem 2.1]{HHHKR} the initial ideal of $J$ is given by $\ini(J)=(x_1y_2, x_2y_3, x_3y_4)$, with respect to the lexicographic term order in which $x_1 > \dots > x_4 > y_1 > \dots > y_4$. Theorem~\ref{main} shows that $x_1+y_2, x_2+y_3, x_3+y_4$ is an initially regular sequence of $R/J$. Note that $x_4$ and $y_1$ are free variables in $R/\ini(J)$ and thus $x_4, y_1$ is a regular sequence on $R/\ini(J)$.
These together give an initially regular sequence of $R/J$ of length 5, namely, $x_4, y_1, x_1+y_2, x_2+y_3, x_3+y_4$. This implies that $\depth R/J \ge 5$. Computation with Macaulay~2~\cite{M2}  indeed shows that $\depth R/J = 5$.
\end{example}

\subsection{Blowup Algebras of Edge Ideals}

In \cite{V}, Villarreal gave an explicit description of the defining ideal, also known as the ideal of equations, of the Rees algebra of the edge ideal of any graph in terms of the primitive even closed walks of the graph. Passing to the fiber cone, $\mathcal{F}$, eliminates the linear syzygies in this ideal of equations, resulting in a toric ideal that is the ideal of equations of $\mathcal{F}$. Moreover, this binomial ideal is known to form a universal Gr\"{o}bner basis for the ideal (see for example \cite{TT}). Given a graph, one can form an associated hypergraph using the even closed walks (and an appropriate term order) and apply the results of this paper to obtain a lower bound on the depth of the fiber cone.

\begin{example}\label{Rees}
Consider the graph 

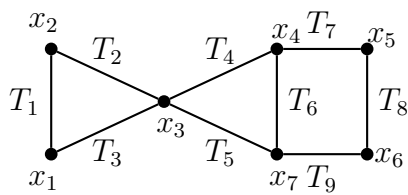
\begin{tikzpicture}[node distance=1cm, on grid, auto]

\tikzstyle{point}=[circle,thick,draw=black,fill=black,inner sep=0pt,minimum width=4pt,minimum height=4pt]

 \node (a)[point,label={[xshift=-0.1cm, yshift=-0.7 cm]:$x_1$}] at (.5,0) {};

\node (b)[point,label={[xshift=-0.1cm, yshift=0.0 cm]:$x_2$}] at (0.5,1.4) {};

    \node (c)[point,label={[xshift=0.1cm, yshift=-0.7 cm]:$x_3$}] at (2,0.7) {};

\node (g)[point,label={[xshift=0.1cm, yshift=-0.7 cm]:$x_7$}] at (3.5,0) {};

\node (f)[point,label={[xshift=0.3cm, yshift=-0.4 cm]$x_6$}] at (4.7,0) {};

\node (e)[point,label={[xshift=0.2cm, yshift=-0.2 cm]:$x_5$}] at (4.7,1.4) {};

\node (d)[point,label={[xshift=-0.3cm, yshift=0.2 cm]1:$x_4$}] at (3.5,1.4) {};

\path[draw, thick]
(a) edge node {$T_1$} (b)
(b) edge node[above] {$T_2$} (c)
(a) edge node[below] {$T_3$} (c)
(c) edge node[above] {$T_4$} (d)
(c) edge node[below] {$T_5$} (g)
(d) edge node {$T_6$} (g)
(d) edge node {$T_7$} (e)
(e) edge node{$T_8$} (f)
(f) edge node {$T_9$} (g);

\end{tikzpicture}

\noindent corresponding to $I=(x_1x_2, x_2x_3, x_1x_3, x_3x_4, x_3x_7, x_4x_7, x_4x_5, x_5x_6, x_6x_7).$
 The ideal of equations of the fiber cone $\mathcal{F}=\mathcal{F}(I)$ is given by $J=(T_1T_4T_5-T_2T_3T_6, T_6T_8-T_7T_9)$. The initial ideal is $\ini(J)=(T_1T_4T_5, T_6T_8)$ with respect to the order $T_1>T_2> \ldots >T_8$.  By Theorems~\ref{main} and ~\ref{iterate repeated neighbors}, $T_2,T_3,T_7,T_9,T_6+T_8,T_1+T_4,T_4+T_5$ is an initially regular sequence, and so $\depth \mathcal{F} \geq 7$. Using Macaulay 2~\cite{M2} it can be verified that the depth is precisely $7$.
\end{example}


\begin{thebibliography}{99}


\bibitem{AB} M. Auslander, D.A. Buchsbaum, Homological dimension in Noetherian rings. Proc. Nat. Acad. Sci. USA, 42 (1956) pp. 36–38.

\bibitem{AL} W.W. Adams and P. Loustaunau, An Introduction to Gr\"obner Bases. GSM 3, American Mathematical Society, 1994.




\bibitem{BH} W. Bruns and J. Herzog, Cohen Macaulay rings. Cambridge Studies in Advanced Mathematics, 39. Cambridge University Press, Cambridge, 1993. xii+403 pp.



\bibitem{CV} A. Conca, M. Varbaro, Squarefree Gr\"{o}bner degenerations, Preprint (2018), \href{https://arxiv.org/abs/1805.11923}
{\tt arXiv:1805.11923.}

\bibitem{DS} H. Dao and J. Schweig, Projective dimension, graph domination parameters, and independence complex homology. J. Combin. Theory Ser. A 120 (2013), 453--469.

\bibitem{DS2} H. Dao and J. Schweig, Bounding the projective dimension of a squarefree monomial ideal via domination in clutters. Proc. Amer. Math. Soc. 143 (2015), no. 2, 555--565.



\bibitem{Eis} D. Eisenbud, Commutative algebra. With a view toward algebraic geometry. Graduate Texts in Mathematics, 150. Springer-Verlag, New York, 1995.



\bibitem{FH} S. Faridi and B. Hersey, Resolutions of monomial ideals of projective dimension 1. Comm. Algebra 45 (2017), no. 12, 5453--5464.

\bibitem{FM} L. Fouli and S. Morey, A lower bound for depths of powers of edge ideals. J. Algebraic Combin. 42 (2015), no. 3, 829--848.


\bibitem{GMSVV} P. Gimenez, J. Mart\'{i}nez-Bernal, A. Simis, R.H. Villarreal, C.E. Vivares, Symbolic powers of monomial ideals and Cohen-Macaulay vertex-weighted digraphs, Preprint (2018)  \href{https://arxiv.org/abs/1706.00126}{\tt arXiv:1706.00126.}


\bibitem{M2}{{D. R. Grayson and M. E. Stillman},{ Macaulay2, a software system for research in algebraic geometry},
         {Available at \url{https://faculty.math.illinois.edu/Macaulay2/}} }

\bibitem{Gro} A. Grothendieck, Cohomologie locale des faisceaux coh\'{e}rents et th\'{e}or\`{e}ms de Lefschetz locaux et globaux. SGA 2 , IHES (1962).

\bibitem{Oaxaca} H.T. H\`{a}, K.-N. Lin, S. Morey, E. Reyes, R.H. Villarreal, Edge ideals of oriented graphs, Int. J. of Algebra and Computation (2019) DOI: 10.142/S0218196719500139.




\bibitem{HH} J. Herzog and T. Hibi, Monomial Ideals. Graduate Texts in Mathematics 260, Springer, 2011.


\bibitem{HHHKR} J. Herzog, T. Hibi, F. Hreinsd\'{o}ttir, T. Kahle, J. Rauh, Binomial edge ideals and conditional independence statements. Adv. in Appl. Math. 45 (2010), 317-333.



\bibitem{HHKO1} T. Hibi, A. Higashitani, K. Kimura and A.B. O'Keefe, Depth of initial ideals of normal edge rings. Comm. Algebra 42 (2014), no. 7, 2908--2922.




\bibitem{KT} K. Kimura, N. Terai, Binomial arithmetical rank of edge ideals of forests. Proc. Amer. Math. Soc. 141 (2013), 1925--1932.


\bibitem{LM2} K.-N. Lin and P. Mantero, Projective dimension of string and cycle hypergraphs. Comm. Algebra 44 (2016), no. 4, 1671--1694.





\bibitem{NPV} J. Neves, M. Vaz Pinto and R. H. Villarreal, Regularity and algebraic properties of certain lattice ideals. Bull. Braz. Math. Soc. (N.S.)  45 (2014), 777--806.


\bibitem{P2} D. Popescu, Upper bounds of depth of monomial ideals. J. Commut. Algebra 5 (2013), no. 2, 323--327.

\bibitem{P3} D. Popescu, Graph and depth of a monomial squarefree ideal. Proc. Amer. Math. Soc. 140 (2012), no. 11, 3813--3822.

\bibitem{TT} C. Tatakis and A. Thoma, On the universal Gr\"obner bases of toric ideals of graphs. J. Combin. Theory Ser. A 118 (2011), no. 5, 1540-1548.

\bibitem{RT} C. Tenter\'{i}a and H. Tapia-Recillas, Reed-Muller Codes: An ideal theory approach, Comm. Algebra 25 (1997) no. 2, 401-413.

\bibitem{Serre} J.-P. Serre, Alg\`ebre locale. Multiplicit\'es. Lect. Notes in Math., 11. Springer, 1965.


\bibitem{V} R.H. Villarreal, Rees algebras of edge ideals. Comm. Algebra 23 (1995), no. 9, 3513-3524.

\end{thebibliography}
\end{document}